\documentclass{amsart}
\usepackage{amssymb, latexsym, amsthm}

\usepackage{mathrsfs} 
\usepackage{mathabx} 

\usepackage[OT2,T1]{fontenc}
\DeclareSymbolFont{cyrletters}{OT2}{wncyr}{m}{n}
\DeclareMathSymbol{\Sha}{\mathalpha}{cyrletters}{"58}
\numberwithin{equation}{section}

\newtheorem{thm}[equation]{Theorem}
\newtheorem{conj}[equation]{Conjecture}
\newtheorem*{conj*}{Conjecture}

\newtheorem{prop}[equation]{Proposition}

\theoremstyle{definition}
\newtheorem{eg}[equation]{Example}
\newtheorem{egs}[equation]{Examples}
\newtheorem{prob}[equation]{Problem}
\newtheorem{defn}[equation]{Definition}

\newcommand{\cZ}{Z}

\newcommand{\ed}{\mathrm{ed}}

\newcommand{\ol}[1]{\overline{#1}}
\newcommand{\op}{\mathrm{op}}
\newcommand{\ot}{\otimes}

\newcommand\ra{{\rightarrow}}
\newcommand{\rk}[1]{\operatorname{rank}_{#1}}
\newcommand{\sep}{{\mathrm{sep}}}

\newcommand{\Alg}{\mathrm{Alg}}
\newcommand\Brp[1][p]{{{}_{{#1}\!\!\!\:}\Br}}
\newcommand{\F}{\mathbb{F}}
\newcommand{\Fields}{\mathsf{Fields}}

\newcommand{\Gm}{\mathbb{G}_m}
\renewcommand{\H}{{\text{\rm H}}}
\newcommand{\OO}{\mathscr{O}}
\newcommand{\NS}{\mathrm{NS}}
\newcommand{\Pic}{\mathrm{Pic}}
\newcommand{\PPic}{\mathscr{P}\mathit{ic}}
\newcommand\R{{\mathbb{R}}}
\newcommand{\Sets}{\mathsf{Sets}}
\newcommand\UD{{\text{\it UD}}}

\newcommand\Z{{\mathbb{Z}}}

\def\tensor{{\otimes}}

\def\Q{{\mathbb{Q}}}

\DeclareMathOperator{\chr}{char}
\DeclareMathOperator{\charac}{char}
\DeclareMathOperator{\ind}{ind}
\DeclareMathOperator{\per}{per}

\DeclareMathOperator{\rdim}{rdim}
\DeclareMathOperator{\trdeg}{trdeg}

\DeclareMathOperator{\Br}{Br}
\DeclareMathOperator{\CP}{CPAlg}
\DeclareMathOperator{\Gal}{Gal}
\DeclareMathOperator{\GL}{GL}
\DeclareMathOperator{\Hom}{Hom}
\DeclareMathOperator{\M}{M}
\DeclareMathOperator{\PI}{Br.\!dim}
\DeclareMathOperator{\PGL}{PGL}
\DeclareMathOperator{\PSL}{PSL}
\DeclareMathOperator{\SL}{SL}
\DeclareMathOperator{\Spec}{Spec}
\DeclareMathOperator{\Spin}{Spin}
\DeclareMathOperator{\SK}{SK}

\long\def\forget#1\forgotten{}

\def\Wik{\subsection*{What is known}}

\usepackage{hyperref}
\hypersetup{pdftitle={Open problems on Central Simple Algebras}}
\hypersetup{pdfauthor={Ketura 2010 Open Problem Session Group}}
\hypersetup{colorlinks=true,linkcolor=blue,anchorcolor=blue,citecolor=blue,letterpaper=true}

\begin{document}

\renewcommand{\thesubsection}{\thesection.\Alph{subsection}}

\title%
{Open problems on Central Simple Algebras}

\author{Asher Auel \and Eric Brussel \and Skip Garibaldi \and Uzi Vishne}

\address{(Auel, Brussel, Garibaldi) Department of Mathematics \& Computer Science, 400 Dowman Drive, Emory University, Atlanta, Georgia, USA 30322}

\address{(Vishne) Department of Mathematics \& Statistics, Bar-Ilan University, Ramat-Gan 52900, Israel}

\subjclass[2000]{16K20, 16K50, 16S35, 12G05, 14F22}


\begin{abstract}
We provide a survey of past research and a list of open problems regarding central simple algebras 
and the Brauer group over a field, intended both for experts and for beginners.
\end{abstract}

\maketitle


There are many accessible introductions to the theory of central
simple algebras and the Brauer group, such as \cite{Albert},
\cite{Hrstn:nc}, \cite{Pierce}, \cite{Draxl}, \cite[Ch.~7]{Rowen},
\cite{Kersten}, \cite{Jac:fd}, and \cite{GilleSz}---or at a more
advanced level, \cite{Salt:lect}. But there has not been a survey of
open problems for a while---the most prominent recent references are
the excellent surveys \cite{Am:surv}, \cite{Salt:rvw}, \cite{Salt:fd},
and \cite{Salt:Amsurv}. Since the last survey, major new threads have
appeared related to geometric techniques. As examples, we mention (in
chronological order): \begin{itemize} \item Saltman's results on
division algebras over the function field of a $p$-adic curve, see
\cite{Salt:ram}, \cite{Saltman:cyclic}, \cite{Br:surv}, \cite{PS:Qpt},
\cite{PS:Qpt2}; \item De Jong's result on the Brauer group of the
function field of a complex surface, see \cite{dJ},
\cite{Lieblich:PI}, and \S\ref{PI.sec}; \item
Harbater--Hartmann--Krashen patching techniques, see \cite{HHK:app} and
\cite{HHK:subfields}; and \item Merkurjev's bounding of essential
$p$-dimension of $\PGL_n$, see \cite{M:p2}, \cite{M:pn}, and
\S\ref{ed.sec}. \end{itemize}

Motivated by these and other developments, we now present an updated
list of open problems. Some of these problems are---or have
become---special cases of much more general problems for algebraic
groups. To keep our task manageable, we (mostly) restrict our
attention to central simple algebras and
$\PGL_n$.
The following is an idiosyncratic list that originated during
the Brauer group conference held at Kibbutz Ketura in January 2010. We
thank the participants in that problem session for their contributions
to this text. We especially thank Darrel Haile, Detlev Hoffmann,
Daniel Krashen, Dino Lorenzini, Alexander Merkurjev, Raman Parimala, Zinovy
Reichstein, Louis Rowen, David Saltman, Jack Sonn, Venapally Suresh, Jean-Pierre
Tignol, and Adrian Wadsworth for their advice, suggestions, and (in some cases)
contributions of text.

\setcounter{tocdepth}{1}
\tableofcontents

\setcounter{section}{-1}
\section{Background}\label{background}

\subsection*{Basic definitions}
Let $F$ be a field.
A ring $A$ is a \emph{central simple $F$-algebra} if $A$ is a finite dimensional $F$-vector
space, $A$ has center $F$, and $A$ has no nontrivial 2-sided ideals.
An \emph{$F$-division algebra} is a central simple $F$-algebra with no nontrivial 1-sided ideals.
Every central simple $F$-algebra $B$ is isomorphic to a matrix algebra over a unique
$F$-division algebra (Wedderburn's theorem), which is said to be {\it associated} with $B$.  
Two algebras are said to be (Brauer) \emph{equivalent} if their
respective associated division algebras are $F$-isomorphic.
Thus two algebras $A$ and $B$ are equivalent if there exist number $n$ and $m$ such that
$\M_n(A)=\M_m(B)$.
The resulting set of equivalence classes forms the (classical) Brauer group $\Br(F)$,
and this set inherits a group structure from the $F$-tensor product on algebras.
The \emph{index} of a central simple algebra is the degree of its associated division algebra,
and the \emph{period} (or \emph{exponent}) is the order of its Brauer class.

Recall that an \emph{\'etale} $F$-algebra is a finite product of finite separable field extensions of $F$.
Such an algebra $L$ is {\it Galois} with group $G$ if $\Spec L$ is a
$G$-torsor over $F$. In other words, $L/F$ is finite and separable, and $G$ acts on $L$
via $F$-automorphisms such that the induced action on $\Hom_{\text{$F$-alg.}}(L,F_\sep)$ is simple
and transitive, where $F_\sep$ is a separable closure of $F$.
If $L/F$ is $G$-Galois then there exists a subgroup $H$ of $G$ and an
$H$-Galois field extension $K/F$ such that $L$ is $G$-isomorphic to $\Hom_{\text{$F$-v.s.}}(F[G/H], K)$ with the induced action and pointwise product.  Equivalently, $L$ is $G$-isomorphic t
to
$\prod_{G/H}(K)_{sH}$, where $G/H$ is the coset space.  Compare \cite[\S{I.5}]{Milne}.

\subsection*{Definition of a crossed product}
The classical theory shows that each central simple $F$-algebra $B$ of degree $n$ contains a 
commutative \'etale subalgebra $T$ of degree $n$ over $F$, called a \emph{maximal torus}.
Equivalently, there exists an \'etale extension $T/F$ of degree $n$ that \emph{splits} $B$,
i.e., such that $B\otimes_F T=\M_n(T)$.
By definition, $B$ is a \emph{crossed product} if $B$ contains a maximal torus that
is \emph{Galois} over $F$.
Further, $B$ is \emph{cyclic} if $B$ contains a maximal torus that is cyclic Galois over $F$.
It is easy to see that every $B$ is Brauer-equivalent to a crossed product.
For if $L/F$ is the Galois closure of a maximal torus $T\subset B$ with respect to
a separable closure $F_\sep$, and $[L:T]=m$, then $L$ is a Galois maximal torus of 
$C=\M_m(B)$.  

Determining the crossed product central simple $F$-algebras $B$ with Brauer-equivalent division algebra $A$ is equivalent to determining the finite Galois splitting fields of $A$.  For $B$ is a $G$-crossed product via some $G$-Galois maximal torus $L/F$ if and only if $A$ is split by an $H$-Galois \emph{field} extension $K/F$, where $H \le G$ and $K/F$ where $K$ is the image of $L$ under an (any) $F$-algebra homomorphism from $L$ to $F_\sep$.
Thus, problems involving crossed products focus on the Galois splitting fields of division algebras.

It follows from the above that if $A$ is a crossed product or is cyclic, then so is every $B = \M_n(A)$. 
The converse turns out to be false; there exist noncrossed products, as we shall see below.
Moreover, in case $F$ has prime characteristic $p$, Albert proved that every $p$-algebra (i.e., central simple algebra of period a power of $p$) is Brauer-equivalent to a cyclic algebra.  But \cite{AmSalt} gives an example of a non-cyclic $p$-algebra.  One can lift this example to a complete discrete valuation ring of characteristic 0 and so find a central simple algebra $A$ over a field of characteristic 0 such that $\M_n(A)$ is cyclic for some $n$ but $A$ is not.

\subsection*{Cohomological interpretation}
Much of the interest in crossed products stems from the fact that 
the Galois structure of a crossed product allows us to describe it simply (see \S\ref{ed.sec})
and explicitly.
For if $B$ is a crossed product via a $G$-Galois extension $K/F$,
then $B$'s algebra structure is encoded in a Galois 2-cocycle
$f\in\cZ^2(G,K^\times)$, which is easily extracted from $B$ in principle.
Conversely, each $f\in\cZ^2(G,K^\times)$ defines a $G$-crossed product $B_f$
with maximal subfield $K/F$, by letting
$$
B_f=\bigoplus_{s\in G}K u_s
$$
be a $K$-vector space on formal basis elements $u_s$,
and defining multiplication by
$u_s u_t=f(s,t)u_{st}$ and $u_s x=s(x) u_s$ for all $s,t\in G$ and $x\in K$.
The equivalence relation on central simple algebras
induces an equivalence on cocycles, and the resulting group of equivalence classes
is the Galois cohomology group $\H^2(G,K^\times)$.
Note that every central simple algebra is a crossed product if and only if
the correspondence between classes extends to the level of cocycles and algebras.

\subsection*{The crossed product problem for division algebras}
Since every algebra representing a given class is a crossed product if and only if its associated
division algebra is a crossed product,
the interesting question is whether every $F$-\emph{division algebra} is a crossed product.

Fields over which all division algebras are known to be crossed products 
include those that do not support division
algebras in the first place (fields of cohomological dimension less than 2, such
as algebraically closed fields and quasi-finite fields),
those over which all division algebras are cyclic (global fields and local fields),
and those over which all division algebras are abelian crossed products (strictly henselian fields \cite[Th.~1]{B:strictly}, at least for algebras of period prime to the residue characteristic).

For much of the 20th century it was reasonable to conjecture that
every central simple algebra was itself a crossed product.  Then, in
1972, Amitsur produced noncrossed product division algebras of any
degree $n$ divisible by $2^3$ or $p^2$ for an odd prime $p$, provided
$n$ is prime to the characteristic.  The centers of these (universal)
division algebras are given as invariant fields; their precise nature
is itself a topic of interest (see \S\ref{center.sec}).  In
\cite{SS73} Schacher and Small extended Amitsur's result to include
the case where the characteristic is positive and does not divide the
degree.  In \cite{Salt:non}, Saltman included the $p$-algebra case,
when $n$ is divisible by $p^3$.  Subsequent constructions of
noncrossed products---none improving on the indexes found by Amitsur---have since appeared in \cite{Risman:cyclic}, \cite{Salt:small},
\cite{Rowen:counter}, \cite{Tignol:cyclic}, \cite{JW:non},
\cite{B:thesis}, \cite[Th.~1.3]{RY:splitting}, \cite{BMT}, and others.
The most explicit construction to date is in \cite{Hanke}.

Responding to the sensitivity of Amitsur's degree $p^2$ examples to the ground field's
characteristic and roots of unity, Saltman and Rowen observed that every division algebra of degree
$p^2$ becomes a crossed product on some prime-to-$p$ extension (the degree $p$ case being trivial), 
whereas for degree $p^3$ and above, 
examples exist that are stable under such extensions (\cite{RowS:prime}); see \cite[Cor.~2.2.2]{McK07} for the $p$-algebra case.

\section{Non-cyclic algebras in prime degree}\label{Q1}

Perhaps the most important open problem in the theory of central simple algebras is:

\begin{prob} \label{cyclic}
Given a prime $p$, construct a non-cyclic division algebra of degree $p$ over some field $F$.
\end{prob}

This was listed as Problem 7 in Amitsur's survey \cite{Am:surv},
and Problem 1 in \cite{Salt:fd}.
Of course, division algebras of non-prime degree need not be cyclic.  This is true already in the smallest possible case of degree 4 (\cite{Albert:deg4eg}).

\Wik Prime-degree division algebras for the two smallest primes are always cyclic:
The $p=2$ case is elementary, and the $p=3$ case was solved by Wedderburn in 1921 (\cite{Wedderburn}). 
A result that should be included with these two is
that any division algebra of degree $n=4$ is a crossed product (\cite{Albert}).
In 1938, Brauer proved that any division algebra of degree $p=5$ has a (solvable) Galois
splitting field of degree dividing 60 (\cite{Brauer5}); this result was improved somewhat in \cite{RowS:semi}.
Despite the promising start, this vexing problem remains completely open for $p>3$.  Specific candidates for prime degree non-crossed products have been proposed in \cite[\S4]{Rowen:palg} and \cite[p.~487]{Vishne}.

It is natural to specialize the problem to fields with known properties.
For some fields $F$, one knows that every division algebra of degree $p$ is cyclic:
when $F$ has no separable extensions of degree $p$ (trivial case), when $F$ is a local field or a global field (by Albert--Brauer--Hasse--Noether), and when $F$ is the function field of a $\ell$-adic curve for $\ell$ prime and different from $p$ \cite{Saltman:cyclic}.  

\subsection*{Generalization} 
In 1934, Albert proved that an $F$-division algebra $A$ of prime degree $p$ is a crossed
product if and only if $A$ contains a non-central element $x$ such that $x^p$ is in $F$, called
a \emph{$p$-central element} (``Albert's cyclicity criterion'', \cite[XI, Th.~4.4]{Albert}).
It is natural to ask whether such elements exist in algebras of degree greater than $p$. When $p = 2$, the degree is 4, and the characteristic is not 2, Albert's crossed product result shows that they do exist. On the other hand, it is shown in \cite{AmSalt} that there exist algebras of degree $p^2$ and characteristic $p$ with no $p$-central elements.  Using this, it is shown in \cite{Sal80} that for $n$ a multiple of $p^2$,
the universal division algebra of degree $n$ over the rational field $\Q$ has no $p$-central elements. When $F$ has a primitive $p$-th root of unity, there are no known examples of division $F$-algebras of degree a multiple of $p$ that do not have $p$-central elements.

\section{Other problems regarding crossed products} \label{Q2}

The central organizing problem in the theory of central simple algebras and the Brauer
group is to determine, possibly for a given field $F$,
the extent to which $F$-central simple algebras are crossed products, and in particular, cyclic crossed products. 
There are still some important unresolved problems in addition to Problem \ref{cyclic}.

\begin{prob} \label{Q2.prob2}
Determine whether $F$-division algebras of prime period are crossed products.  
\end{prob}

Both Problems \ref{cyclic} and \ref{Q2.prob2} are related to the problem of computing the essential dimension of algebraic groups described in \S\ref{ed.sec}.

\subsection*{What is known}
This problem is even more open than Problem \ref{cyclic},
in the sense that no noncrossed products of prime period have
ever been discovered, and no one has shown that all division algebras of a fixed prime period
must be crossed products.  
For odd primes $p$ the problem is completely open.  (The proof of Theorem 6 in \cite{Row82} concerning the existence of noncrossed products
of prime period is incorrect.
In the modified version that appeared later as \cite[Th.~7.3.31]{Row:RT2},
the flaw appears in the last paragraphs of page 255.)

Division algebras of $(\text{period, index})=(2,2^n)$ are known to be crossed products for $n\leq 3$,
and the problem for $n>3$ is open.
Albert proved the $(2,4)$ case (``Albert's Theorem'') 
in \cite[Th.~6]{Albert:deg4}---see alternatively \cite[5.6.9]{Jac:fd}---and
in a second proof \cite{Albert:deg4eg}, he showed that such algebras need not be cyclic.
Rowen proved the $(2,8)$ case for any field in \cite{Rowen:CSA} and \cite{Rowen:22}, see also \cite[5.6.10]{Jac:fd}.
Rowen \cite{Rowen:counter} and later Tignol \cite{Tignol:cyclic} proved that Saltman's universal division algebra of type $(4,8)$ is a noncrossed product.

Regarding specific fields,  
Brussel and Tengan recently proved that division algebras of type $(p,p^2)$ for
$p\neq\ell$ over the function field of an $\ell$-adic curve are crossed products in \cite{BT10}
(cf.\ Saltman's results on these fields, mentioned in the previous section).

As mentioned above, the smallest non-crossed product $p$-algebras known have degree $p^3$.  Therefore even the following is open:

\begin{prob}
Determine whether $p$-algebras of degree $p^2$ are crossed products.
\end{prob}

In view of the fact that all known examples of non-crossed product algebras occur over fields of cohomological dimension at least 3, it is natural to ask:
\begin{prob}
If $F$ has cohomological dimension at most 2, is every central simple $F$-algebra a crossed product?
\end{prob}

The answer to this question is unknown, even when $F$ is the function field of a complex surface.

\subsection*{Smallest Galois splitting fields}
One way to study the failure of a division algebra to be a crossed product
is to determine,
for a given division algebra $A$,
the smallest Brauer-equivalent crossed product(s) $B=\M_r(A)$.
As noted already, it is equivalent
to determine the finite Galois splitting fields for $A$.
As Brauer--Hasse--Noether knew in the 1920's, ``minimal'' Galois splitting fields
can have arbitrarily large degree (Roquette cites \cite{B-N} and \cite{Hasse} in \cite[Chap.~7.1]{Roq:BHN}), 
so it is important to specify that the degree (rather than the splitting field) be minimal.

To simplify the exposition,
we say a finite group $G$ \emph{splits} a central simple $F$-algebra $A$
if there exists a $G$-Galois field extension of $F$ that splits $A$
(equivalently, if some $\M_m(A)$ is a $G$-crossed product).
Since a splitting field of $A$ may or may not contain a maximal subfield of $A$, the problem with
respect to division algebras divides into parts.

\begin{prob} \label{Q2.prob3}
Determine the smallest splitting group(s) of an $F$-division algebra $A$ arising
from the Galois closure of a maximal subfield of $A$.  Determine in particular whether 
every $A$ of degree $n$ has a separable maximal subfield whose Galois closure has group the alternating group $A_n$.
\end{prob}


\begin{prob} \label{Q2.prob5}
Determine whether every division algebra is split by an abelian group,
and find degree bounds over specific fields.
\end{prob}

\Wik  
Both problems are related to Problem 5 in \cite{Am:surv},
and Problem~\ref{Q2.prob5} is Question 1 in \cite{Am:surv}.
In the situation where $A$ is a crossed product,
Problem \ref{Q2.prob3} becomes a problem of \emph{group admissibility}.
It then has a strong number and group-theoretic character, 
and we put it in Problem~\ref{sch.prob} below.

If $A$ has degree $n$, then $A$ contains a maximal separable
subfield (of degree $n$) that is contained in an $S_n$-Galois maximal torus,
which bounds the answer to the first part of Problem \ref{Q2.prob3} from above.
It is clear that these are in general not hard bounds,
since if $n=3$ then $C_3$ splits $A$ by Wedderburn's results.
It is unknown whether there exists a (noncrossed product) division algebra of degree $n > 4$
that is split by no proper subgroup of $S_n$, 
prompting Amitsur to ask Problem \ref{Q2.prob3}.

With the additional structure provided by an involution of the first kind,
some positive results for Problem~\ref{Q2.prob3} can be proved in the period $2$ case, for $\chr F\ne 2$.
In particular, if $A$ has period $2$ and $2$-power degree $2m$, one can prove the existence
of a subfield of degree $m$ in a maximal separable subfield of $A$, and it follows that
the Galois closure has group $H=C_2^m\rtimes S_m$.
Using the main result of Parimala--Sridharan--Suresh
in \cite{PSS:disc},  one can improve this to $\bar H=C_2^{m-1}\rtimes S_m$; we leave the details to the interested reader.
Note that $H$ is not contained in $A_{2m}$, but $\bar H$ is.

Amitsur proved that if the universal division algebra (defined in \S\ref{center.sec} below) $\UD(k,n)$ over an infinite field $k$
is split by a group $G$,
then every division algebra over a field $F$ containing $k$ is split by (a subgroup of) $G$
(unpublished, see \cite[p.15]{Am:surv} and \cite[7.1]{TigAm}).
Thus Problem~\ref{Q2.prob5} is intimately connected to the problem of splitting
the universal division algebra, which is Problem 8 in \cite{Am:surv}.
In \cite{Am:surv} Amitsur remarked that if $n$ is the composite of $r$ relatively prime numbers $n_i$,
then $S_{n_1}\times\cdots\times S_{n_r}$ splits $\UD(k,n)$.
Then in \cite[Th.~7.3]{TigAm} Tignol and Amitsur established a remarkable lower bound on the order of a splitting
group of $\UD(k,n)$, for infinite $k$, and $n$ not divisible by $\chr k$.
See \cite[Cor.~9.4]{TA86} and \cite[Th.~1.3]{RY:splitting} 
for improvements and alternative proofs of this bound in the case where $n$ is a prime power.

If $\chr F$ does not divide the period $n$ of an $F$-division algebra $A$,
the Merkurjev--Suslin theorem shows that 
$A$ has a meta-abelian splitting field,
obtained by adjoining $n$-th roots of unity, if necessary, and then a Kummer extension.
This theorem solved Amitsur's Question 3 in \cite{Am:surv}, and it solves 
the first part of Problem~\ref{Q2.prob5} if $F$ contains the $n$-th roots of unity.
Results bounding the ``symbol length'' over a particular field
provide upper bounds for the degree of abelian (or meta-abelian) Galois splitting fields,
and for this we refer to Problem~\ref{lengthbounds}.
Tignol and Amitsur submitted a lower bound for the order of an abelian
splitting group in \cite[7.5]{TigAm}.

When $F$ does not contain $n$-th roots of unity, 
Problem~\ref{Q2.prob5} has a similar relationship to the more general open problem of determining
whether the Brauer group is generated by cyclic classes, and for this we refer to \S\ref{cyclic.gen.sec}.  
This latter problem was Question 2 in \cite{Am:surv}.
The problem is settled for the 2-, 3-, and 5-torsion subgroups of $\Br(F)$.

\subsection*{$p$-algebras} Somewhat more is known about Problem~\ref{Q2.prob5} for $p$-algebras:
If $\chr F=p$ and $A$ is an $F$-division algebra of $p$-power degree, then 
Albert showed that $A$ has a cyclic splitting field in \cite{Albert:cyclics}, see also \cite[Ch.~VII, Th.~9.31]{Albert}.
Louis Isaac Gordon proved the existence of noncyclic $2$-algebras (of period and index $4$) in \cite{Gordon}
and Amitsur and Saltman proved the existence of (generic) non-cyclic division $p$-algebras
of every degree $p^n$ ($n\ge 2$) in \cite{AmSalt}.
These results raised questions about bounds of cyclic splitting fields.

Albert's methods impose such bounds on $p$-algebras $A$ with fixed ($p$-power) period $q$,
provided $F$ is finitely generated of transcendence degree $r$ over a perfect field $k$.
For then $F^{1/q}$ splits $A$ by \cite[VII.7, Theorem 2]{Albert}, and since $[F^{1/q}:F]=q^r$ 
(\cite[A.V.100, Proposition 4]{Bou:alg2} and \cite[A.V.135, Corollary 3]{Bou:alg2}),
$A$ is similar to a tensor product of $r$ cyclic $p$-algebras of degree $q$ by \cite[VII.9, Theorem 28]{Albert}, hence $A$ is split by a cyclic extension of degree $q^r$, by \cite[VII.9, Lemma 13]{Albert}.
The center $Z(k,n)$ of generic degree-$n$ matrices over a field $k$ is
known to be finitely generated (\cite[Cor.~14.9]{Salt:lect});
if $k$ is a finite field of characteristic $p$, and $Z(k,n)$ has transcendence degree $m$
over $k$, then it follows that the universal division algebra
$\UD(k,p^n)$ is split by a cyclic field extension of degree $p^{mn}$,
and hence for any field $F$ containing $k$, any $F$-division algebra $A$ of degree $p^n$ is
split by a cyclic \'etale extension of degree $p^{mn}$ over $F$, by \cite[p.15]{Am:surv}.

Because of this connection between the degree of a cyclic splitting field
of a $p$-algebra and the transcendence degree of the center of generic matrices over the prime field,
Problem~\ref{Q2.prob5} for $p$-algebras is closely tied to \S\ref{ed.sec} on essential dimension.
As for known bounds, we have $m \leq p^{2n}+1$ from general principles, and  if $p$ is odd and $p^n\geq 5$, then $m\leq(1/2)(p^n-1)(p^n-2)$ by \cite{LRRS}.
(Note the $p^n=2,3$ cases are settled by Wedderburn's cyclicity results.)

\subsection*{Descending from prime-to-$p$ extensions} \label{PIp.trdeg}
Since every division algebra has a maximal separable subfield,
every division algebra of prime degree $p$
is a crossed product after scalar extension to a prime-to-$p$ field extension.
This has led to some attempts at ``descent'':
If $A$ is an $F$-division algebra
of prime degree $p$, and $K/F$ is a Galois splitting field for $A$
whose group $G$ is a semidirect product $G=C_p\rtimes H$ with $C_p$ normal in $G$
and $H$ of order $m$ prime-to-$p$,
then $A\otimes_F K^{C_p}$ is a $C_p$-crossed product division algebra,
and we can ask whether the Galois structure descends to $A$.
This idea goes back to Albert (\cite{Alb:38}).
In the particularly interesting case where $H=C_m$ and $K^{C_p}=F(\zeta_p)$ 
for a primitive $p$-th root of unity $\zeta_p$,
we call $A$ a \emph{quasi-symbol}, after \cite{Vishne}.  
Note in this case $\chr F\neq p$.

The problem below is strongly related to Problems 3, 4, and 5 in \cite{Salt:fd}.
As Saltman noted in \cite{Salt:fd}, showing such crossed products are cyclic 
is equivalent to showing that the $G$-crossed
product $\M_m(A)$ is also a $G'$-crossed product, where $G'$ is a group
that has $C_p$ as an image.
As in \cite{TigAm}, 
write $G\Rightarrow G'$ if every $G$-crossed product is necessarily a $G'$-crossed product,
and $G\Rightarrow_F G'$ if $G\Rightarrow G'$ for $G$-crossed products whose centers contain $F$.

\begin{prob} \label{Q2.prob6}
Determine groups $G$ and $G'$ and conditions on $F$ under which $G\Rightarrow_F G'$.
In particular: 
\begin{enumerate}
\item 
Determine, for $m$ prime-to-$p$, whether $C_p\rtimes C_m\Rightarrow_F C_p\times C_m$.
\item 
Determine whether all quasi-symbols are cyclic.
\end{enumerate}
\end{prob}

\Wik
Background on the first part of Problem~\ref{Q2.prob6}, as well as a basic
symmetry result for abelian groups, may be found in \cite{TigAm}.
We will say that a finite group $G$ {\it splits} an $F$-algebra $A$ if $A$ is split
by a $G$-Galois extension of $F$.
Note that in \ref{Q2.prob6}(1), $C_p\rtimes C_m\Rightarrow_F C_p\times C_m$ if and only if
all $F$-algebras $A$ of prime degree $p$ that are split by $C_p\rtimes C_m$ are cyclic.

In 1938 Albert showed that if $\chr F=p$, and $A$ is a $p$-algebra of degree $p$ split by a
Galois extension whose group is of the form $C_p\rtimes C_m$ for $m$ prime-to-$p$,
then $A$ is cyclic (\cite{Alb:38}), hence $C_p\rtimes C_m\Rightarrow_F C_p\times C_m$ in this case.
He actually proved more generally that for $p$ dividing $n$ and arbitrary $m$, any
$C_n\rtimes C_m$-crossed product is abelian, again for $\chr F=p$
(see \cite{AravireJacob} for a different proof).
In 1999 Rowen 
proved that if $A$ is a $p$-algebra of odd degree $p$ that
becomes equivalent to a certain type (``Type A'') of $p$-symbol
over a quadratic extension $E/F$, then $A$ is already cyclic over $F$ (\cite[Th.~3.4]{Rowen:palg}).
This generalized Wedderburn's degree 3 theorem (in characteristic 3) and 
proved the cyclicity of a certain $p$-algebra of degree $p$ studied by Albert.
Rowen constructed a generic $p$-algebra of degree $p$ that becomes a symbol over a quadratic
extension, and conjectured it to be noncyclic for $p>3$
(\cite[\S4]{Rowen:palg}).

In 1982, Rowen and Saltman
showed that if $n$ is odd, $F$ contains a primitive $n$-th root of unity,
(so $\chr F$ does not divide $n$), and $A$ is a division algebra of degree $n$ that
is split by the dihedral group $D_n=C_n\rtimes C_2$ of order $2n$, then $A$ is cyclic (\cite{RowS:dih}).
Thus $D_n\Rightarrow_F C_n\times C_2$ in this case.
Their techniques generalized those used by Albert to prove cyclicity in degree three, in \cite{Albert}.
Mammone and Tignol proved the same result using different methods in \cite{MammoneTig},
and Haile gave another proof in \cite{Haile:dih}.
In 1996, Haile, Knus, Rost, and Tignol proved that when $\chr F\neq 2$,
$n$ is odd, and $\mu_n=\Z/n(t)$ for some $t\in F^\times$,
then any $F$-division algebra of degree $n$ split by $D_n$ is cyclic (\cite[Cor.~30]{HKRT}).
Vishne proved this more generally for $n$ odd and $[F(\mu_n):F]\leq 2$, 
while removing the restrictions on $\chr F$ (\cite[Th.~13.6]{Vishne}).
In the $n=5$ case, Matzri showed how to adapt Rowen and Saltman's proof to remove
the roots of unity and characteristic hypotheses (\cite{Matzri:deg5}),
before using results of \cite{Vishne} 
to prove that $\Brp[5](F)$ is generated by cyclic classes when $\chr F\neq 5$.

In 1983 Merkurjev, in a paper that heavily influenced \cite{Vishne} and \cite{Matzri:deg5},
proved a criterion for quasi-symbols of prime degree to be cyclic.  
To state this result, we fix some notation and hypotheses.
Let $p$ be a prime, $n=p^r$, $\zeta_n$ a primitive $n$-th root of unity, $E=F(\zeta_p)$,
$\nu:\Gal(E/F)\to(\Z/p)^\times$ the canonical map, and $\varphi:\Gal(E/F)\to(\Z/p)^\times$ a character.
Then consider the hypothesis
\begin{equation} \label{Q2.star}
\parbox{4.2in}{$\zeta_n\in E$ and $K/F$ is a $C_n\rtimes_\varphi C_m$-Galois extension, with $E=K^{C_n}$.}
\end{equation}
In this notation, a quasi-symbol of degree $p$
split by $K$ is said to be of {\it type} $(\varphi,\varphi')$, where $\varphi'=\nu\varphi^{-1}$.
Merkurjev's criterion is that a quasi-symbol of degree $n=p$ is cyclic if and only
if it is of type $(\nu,1)$ or $(1,\nu)$ (\cite{M:Brauer}).
Vishne, expanding on Merkurjev's framework, 
proved that Merkurjev's cyclicity criterion holds for quasi-symbols of degrees $n=p^r$ under \eqref{Q2.star},
and constructed a generic quasi-symbol, which he conjectured to be noncyclic.
He also proved that the type of a quasi-symbol (of degree $n=p^r$)
is symmetric in $\varphi$ and $\varphi'$, so that a quasi-symbol split by
$C_n\rtimes_\varphi C_m$ is also split by $C_n\rtimes_{\varphi'}C_m$.

In 1996 Rowen and Saltman proved that if $F$ contains a $p$-th root of unity,
and $A$ is a division algebra of degree $p$ that is split by the semidirect product
$C_p\rtimes C_m$ for $m=2,3,4$, or $6$ and dividing $p-1$, then $A$ is cyclic.
Thus $C_p\rtimes C_m\Rightarrow_F C_p\times C_m$ in this case.
Vishne extended this result for quasi-symbols of degree $n=p^r$ under \eqref{Q2.star}, 
for any $m$ dividing $p-1$ (\cite{Vishne}).

If $m$ divides $n$ and $F$ contains the $m$-th roots of unity, then it is elementary to show
$C_n\Rightarrow_F C_{n/m}\times C_m$ (see \cite{Vishne:dih} for proof).
This is best possible, since there are division algebras $D/F$ where the absolute Galois group of $F$ is rank 2 abelian.
As for $p$-algebras,
Saltman proved in \cite{Salt:spl} that if $\chr F=p$ and $n$ is a power of $p$,
then $C_n\Rightarrow_F G$ for any group $G$ of order $n$.
Finally, Vishne proved in \cite{Vishne:dih} that if $\chr F$ does not divide $n$
and $F$ contains the $n/m$-th roots of unity, then 
$D_n\Rightarrow_F D_m\times C_{n/m}$ holds for classes of period 2
(\cite{Vishne:dih}).

We say a group $G$ is {\it rigid} (resp.~\emph{rigid over $F$})
if there exists a central simple algebra (resp.~\emph{over $F$}) that is a crossed product
with respect to $G$ and $G$ only.
A central simple algebra can be a crossed product with respect to many different groups;
Schacher showed in \cite{Schacher:th1} that for any number $n$ there exists a number field $F$
and a central simple $F$-algebra $A$ that is a crossed product with respect to every group of
order $n$ (see Problem~\ref{sch.prob}).  
At the other extreme, it is obvious that cyclic groups of prime order are rigid.
Amitsur (\cite{Am:central}) proved that elementary abelian groups are rigid, and used this fact
in his proof that the generic division algebra is not a crossed product.
Saltman and Amitsur--Tignol proved every noncyclic abelian group is rigid in \cite{Salt:non} and
\cite{TigAm}.
For a time there were no known nonabelian rigid groups, and their existence
is Problem 4 in \cite{Salt:fd}.  
However in 1995 Brussel gave examples of rigid nonabelian groups over the fields $\Q(t)$, 
proving that for every prime $p$ the group 
$$
G=C_{p^2}\rtimes C_p= \{x,y \mid |x|=p^2,\ |y|=p,\ yxy^{-1}=x^{p+1} \}
$$ 
is rigid (\cite{B:thesis}),
and later that when $p$ is odd, another semidirect product $G=C_{p^3}\rtimes C_p$
is rigid (\cite{B:embed}).
All of Brussel's examples depend on the absence of roots of unity in the ground field,
and it is unknown whether the same results hold without this assumption.
A rigid group of type $C_p\rtimes C_m$ with $m$ prime-to-$p$ would settle Problem \ref{cyclic}.

\subsection*{Footnote}
In the end, the goal of the problems described in this section is to
determine the nature of Galois splitting fields of smallest degree for
cohomology classes of various kinds. But from an algebra-theoretic
point of view, the importance of the crossed product problem is due at
least partly to a long history of
attempted solutions, together with
a sentiment expressed by
A.A.~Albert, who spent much of his career studying it.
\begin{quote}
The importance of crossed products is due not merely to the fact
that up to the present they are the only [central] simple algebras which have actually been 
constructed but also to {\rm Theorem 1}:  Every [central] simple algebra is similar to a 
crossed product.  --- Albert, 1939 \cite[\S{V.3}]{Albert}
\end{quote}

\section{Generation by cyclic algebras} \label{cyclic.gen.sec}

By far the most elementary central simple algebras known are the cyclic algebras,
which were the first to be studied, and 
which to this day have provided an indispensable tool for investigating
central simple algebras and the Brauer group.  
Certainly 
the simplest expression of a Brauer class that one could hope for
would be as a sum of classes of cyclic algebras,
and in fact most (if not all) known descriptions of
the Brauer group are direct sum decompositions into subgroups
consisting of cyclic classes (or \emph{cyclic factors}).  For example,
by Auslander--Brumer and Fadeev's theorem (see \cite[Theorem 9.2]{GMS}), the
Brauer group of a rational function field in one variable is a direct
sum of (corestrictions of) cyclic factors.  Also, many
valuation-theoretic results on noncrossed products and indecomposable
algebras rely on Witt's theorem (see \cite[7.9]{GMS}), which presents the Brauer
group of a complete discretely valued field of rank one as a direct
sum of two cyclic factors.
For these reasons, one of the most important practical problems in the
study of the Brauer group is to determine whether or not the Brauer
group is generated by the classes of cyclic algebras.
This appears as Question 2 in \cite{Am:surv}.
Authors who studied this problem in the 1930's in one form or another 
include Albert (\cite{Albert:cyclics}), Teichm\"uller (\cite{Tei36}), 
Nakayama (\cite{Nak38}), and Witt (\cite{Witt37}).
For local and global fields, the question has an affirmative answer by
the well-known theorems of Hasse and Albert--Brauer--Hasse--Noether, respectively.
There are no known counterexamples.

After Amitsur's work on noncrossed products, an affirmative answer for
general $F$ might well have seemed unlikely.  But in a seminal
breakthrough, Merkurjev and Suslin proved in 1983 that if $F$ contains
the $n$-th roots of unity, then $\Brp[n](F)$ is generated by the
classes of (cyclic) symbol algebras of degree $n$, via the norm
residue homomorphism in $K$-theory (\cite{MS:Kcoh}).  
The essence of their result was that when $F$ contains the $n$-th roots of unity,
the composition
$$\H^1(F,\mu_n)\otimes \H^1(F,\mu_n)\to \H^2(F,\mu_n^{\otimes 2})\to \H^2(F,\mu_n) = \Brp[n](F)$$
is surjective, where the first map is the cup product, and the second map is cup product with a chosen
generator of $\H^0(F,\mu_n^{\otimes(-1)})=\Hom(\mu_n,\Z/n)$.
Thus $\Brp[n](F)$ is ``generated in degree one''.
This work strongly suggests the following natural question, which was posed
in the 1930's by Albert:

\begin{prob}\label{Gen2} {\textrm{\cite[p.~126]{Albert:cyclics}}}
Is the $n$-torsion subgroup $\Brp[n](F)$ generated by the classes of cyclic algebras of degree (dividing) $n$?
\end{prob}

Problem \ref{Gen2} has an affirmative answer in the following
important cases:
\begin{enumerate}
\item \emph{$n$ is a power of the characteristic of $F$.}  This
follows from \cite[VII.7, Theorem 2]{Albert} and \cite[VII.9, Theorem
28]{Albert}.

\item \emph{$n$ divides $30$.}  By the divisibility remark above, it
suffices to note that the answer is ``yes'' for $n = 2$ \cite{M:mt},
$n=3$ by \cite{M:Brauer}, and $n = 5$ by \cite{Matzri:deg5}.

\item \emph{$n$ is prime and adjoining a primitive $n$-th root unity
yields an extension of degree $\le 3$} by \cite{M:Brauer}.  This leads
to the results summarized in the previous item.

\item \emph{$F$ contains a primitive $n$-th root of unity.}  This was pointed out above.
The Merkurjev-Suslin theorem, which establishes an isomorphism between
$K^M_2(F) / nK^M_2(F)$ and $\H^2(F,\mu_n^{\otimes 2})$, is now
a special case of the recently-proved Bloch--Kato Conjecture.
\end{enumerate}

The smallest open case for $n$ prime is therefore:
\begin{prob}\label{Gen7}
Is $\Brp[7](F)$ generated by cyclic algebras of degree $7$ when 
$\chr F\neq 7$ and $F$ does not contain a primitive $7$-th root of unity?
\end{prob}

(These problems are related to the following very
special question stated by M.~Mahdavi-Hezavehi: \emph{If the
multiplicative group $F^\times$ is divisible, does it follow that
$\Br(F)$ is zero?}  The answer is ``yes'' if $\Br(F)$ is generated by
cyclic algebras, because cyclic algebras are split over such a field.
Therefore the answer is also ``yes'' if $F$ has characteristic zero,
because in that case $F$ has all the roots of unity so the Merkurjev--Suslin theorem applies.)

For composite $n$, even the following is unclear:
\begin{prob}
Is $\Brp[4](F)$ generated by cyclic algebras of degree $4$ when $\charac F \neq 2$ and $\sqrt{-1} \not \in F$?
\end{prob}

And one can ask the weaker question:
\begin{prob} \label{gen.deg}
Is $\Brp[n](F)$ generated by algebras of degree $n$?
\end{prob}

When $n$ is prime, the answer to Problem \ref{gen.deg} is ``yes'' by
\cite{M:Brauer}.  More generally, Merkurjev proved in \cite{M:struct}
that if $n$ is a power of an odd prime $p$ not equal to $\chr F$, or if
$n$ is a power of $2 =: p$ and $\sqrt{-1}\in F$, then $\Brp[n](F)$ is
generated by quasi-symbols (see \S\ref{Q2.prob6}) whose index is
bounded by $n^{n/p}$.  
This result reduces Problem \ref{Gen2} to the determination of whether classes of
quasi-symbols are generated by cyclic classes.  Problem~\ref{gen.deg}
is already open in the case $n = 4$.

\smallskip

\subsection*{Symbol length} \label{length} Once it is known that
$\Brp[n](F)$ is generated by classes of cyclic algebras, the focus
turns to the 
the minimal number of cyclic algebras (of fixed
degree) needed to represent it in the Brauer group.

Suppose $m$ and $n$
have the same prime factors, $m\divides n$,
and $\Brp[m](F)$ is generated by cyclic algebras.  Let $\ell_F(m,n)$
denote the minimal number such that every central simple $F$-algebra
of degree $n$ and period (dividing) $m$ is similar to a tensor
product of $\ell_F(m,n)$ cyclic algebras of degree $m$, unless no such bound exists,
in which case set $\ell_F(m,n)=\infty$.
If $F$ contains the $m$-th roots of unity,
the generic division algebra can be used to show that
$\ell_F(m,n)$ is finite, by the Merkurjev--Suslin theorem.
In this case, we may assume $m$ is a power of a prime $p$ by \cite[Th.~2.2]{Tig84},
and $\ell_F(m,n)$ is known as the {\it symbol length}.

\begin{prob}\label{lengthbounds}
Suppose $p$ is a prime, $r\leq s$, and $\Brp[p^r](F)$ is generated by
cyclic algebras.  Compute $\ell_F(p^r,p^s)$.
\end{prob}

Basic background on this problem can be found in \cite{Tig84}.
Almost all of the known results for $\ell_F(m,n)$
are either for $p$-algebras, or fields $F$ containing the $m$-th roots of unity.
Of course, $\ell_F(p,p)=1$ for $p=2,3$ by the classical theory, for
any field $F$.  Albert's theorem shows that $\ell_F(2,4)=2$, and
$\ell_F(2,8)=4$ when $\chr F\neq 2$ by \cite[Th.~2.6]{Tig84}.
For proofs of these and other results that follow more or
less immediately from earlier work, see \cite{Tig84}.  
In 1937, Teichm\"uller proved that for $p$-algebras, $\ell_F(p^r,p^s)\leq
p^s!\,(p^s!-1)$ (\cite{Teich:zerf}).

Lower bounds can be obtained via results on indecomposable algebras,
which we treat in \S\ref{tensor.sec}.  In \cite[Th.~2.3]{Tig84},
Tignol proved that if $p$ is a prime and $F$ is a field containing the
$p^r$-th roots of unity, then $\ell_F(p^r,p^s)\geq s$, by showing that
the generic division algebra itself cannot be represented as a sum of
fewer than $s$ classes of cyclic algebras of degree $p^r$.  Tignol's
lower bounds were improved by Jacob in \cite{Jacob:indec} in the prime
period case, via the construction of indecomposable algebras of
period $p$ and index $p^n$ for all $n\geq 1$, except, of course, for
$p=2$ and $n=2$.  His bounds, valid for $F$ containing a primitive $p$-th root of unity, are
$\ell_F(p,p^s)\geq 2s-1$ for $p$ odd, and $\ell_F(2,2^s)\geq 2s-2$
(\cite[Remark 3.7]{Jacob:indec}).

For upper bounds, we have
then $\ell_F(p,p) \leq \frac{1}{2}(p-1)!$ for any odd prime $p$, provided $F$ contains
the $p$-th roots of unity, by \cite[Th.~7.2.43]{Row:RT2}.
Kahn established some upper bounds for $\ell_F(2,2^s)$ in
\cite[Th.~1]{BK00} for fields of characteristic not $2$, and conjectured the bound $\ell_F(2,2^s)\leq
2^{s-1}$, which holds in the known cases listed above.  Finally,
Becher and Hoffmann proved that if $F$ contains a $p$-th root of unity and satisfies the (admittedly
strong) hypothesis $[F^\times:F^{\times p}]=p^m$, then $\ell_F(p,p^s)=
m/2$ for $p$ odd or $p=2$ and $F$ non-real.  If $p=2$ and $F$ is real,
the result is $(m+1)/2$ (\cite{BH04}).

In a related result, Mammone and Merkurjev showed that if $A$ is a
$p$-algebra of period $p^r$ and degree $p^s$, and $A$ becomes cyclic
over a finite separable field extension, then $A$ can be represented
by at most $p^{s-r}$ classes of cyclic algebras
(\cite[Prop.~5]{MM91}).

In \cite{Salt:ram}, Saltman proved that if
$F$ is the function field of an $l$-adic curve, then the degree of any $F$-division
algebra divides the square of its period, and it follows by Albert's theorem that if
$l \neq 2$, then $\ell_F(2,4) = 2$.  If $l\neq p$ and $F$ contains the $p$-th roots of unity, then
Suresh proved that $\ell_F(p,p^2)=2$ for odd $p$ in \cite{Sur:symlen}.  
Brussel and Tengan removed the roots of unity requirement in \cite{BT10}, using a different
method.  The 
bounds in the $p=2$ case were crucial to the determination of the $u$-invariant of
function fields of $l$-adic curves by Parimala--Suresh \cite{PS:Qpt2},
see also \cite{PS:Qpt} and \cite{PSur:length}.

The first open case of Problem~\ref{lengthbounds} (with or without roots of unity) is:

\begin{prob}
Find an upper bound on $\ell_F(p,p^2)$ and $\ell_F(p^2,p^2)$ for $p \geq 3$.
\end{prob}

\section{Period-index problem} \label{PI.sec}

By the basic theory, the period of a central simple algebra divides its index (i.e., $\per(A) \mid \ind(A)$ for all $A$), and the two numbers have the same prime factors.  For a field $F$, define the \emph{Brauer dimension} $\PI(F)$ to be the smallest number $n$ such that $\ind(A)$ divides $\per(A)^n$ for every central simple $F$-algebra $A$; if no such $n$ exists, set $\PI(F) = \infty$.

\begin{eg}
If $\Br F = 0$, then trivially $\PI(F) = 0$.  For $F$ a local field or a global field $\PI(F) = 1$ by Albert--Brauer--Hasse--Noether.  For $F$ a field finitely generated and of transcendence degree 2 over an algebraically closed field, $\PI(F) = 1$ by \cite{dJ} and \cite[Th.~4.2.2.3]{Lieblich:PI}.
In case $\PI(F) = 1$, one says that ``$F$ has $\text{period} = \text{index}$''.
\end{eg}

One can focus this notion on a particular prime $p$.  Define $\PI_p(F)$ to be the smallest number $n$ such that $\ind(A)$ divides $\per(A)^n$ for every central simple $F$-algebra $A$ whose index is a power of $p$.  If no such $n$ exists, we put $\PI_p(F) = \infty$.

\begin{egs}
\begin{enumerate}
\item M.~Artin conjectured in \cite{Artin:BS} that $\PI(F) = 1$ for every $C_2$ field $F$.  He proved that $\PI_2(F) = \PI_3(F) = 1$ for such fields, but no more is known.

\item For $F$ finitely generated and of transcendence degree 1 over an $\ell$-adic field $\PI_p(F) = 2$ for every prime $p \ne \ell$ by \cite{Salt:ram}.  Is $\PI_\ell(F)$ finite?

\item If $F$ is a complete discretely valued field with residue field $k$ such that $\PI_p(k) \le d$ for all primes $p \ne \chr(k)$, then $\PI_p(F) \le d + 1$ for all $p \ne \chr(k)$ by \cite[Th.~5.5]{HHK:app}.
\end{enumerate}
\end{egs}

\begin{eg}\label{badprime}
If $F$ has characteristic $p$ and transcendence degree $r$ over a perfect field $k$, then $\PI_p(F) \le r$
by the argument in the discussion on page \pageref{PIp.trdeg} for $p$-algebras.
\end{eg}

\begin{prob} 
If $F$ is finitely generated over a field $F_0$ with $\PI(F_0)$ finite, is $\PI(F)$ necessarily finite? 

It is natural to start by considering only certain fields $F_0$.  For example, suppose that $F$ has prime characteristic $p$ and we take $F_0$ to be its prime field, so $\PI(F_0) = 0$.  If $F$ has transcendence degree 1 over $F_0$, then $F$ is global and $\PI(F) = 1$.  If $F$ has transcendence degree 2 over $F_0$, then $\PI(F) \le 3$ by \cite{Lieblich:PI}.  
\end{prob}

\begin{prob}
Define a notion of ``dimension'' for some class of fields $F$ such that 
\[
\PI(F) \le \dim F - 1 \quad \text{and} \quad \dim(F(t)) = \dim F + 1.
\]
One possibility would be to set a $C_i$ field to have dimension $i$.  The notion of cohomological dimension is obviously not the right one in view of Merkurjev's example of a field $F$ with cohomological dimension 2 and $\PI(F) = \infty$ from \cite{M:simple}.
\end{prob}

\section{Center of generic matrices} \label{center.sec}

Amitsur's universal division algebra has already appeared several times in these notes.  In this section, we discuss some interesting questions regarding its center.  

We begin with the definition.
Let $F$ be a field, let $V=\M_n(F)\oplus \M_n(F)$,
and let $F(V)$ be the field of rational functions on $V$.
For $k=1,2$, let $x_{ijk}$ be the coordinate function defined by the standard elementary
matrix $E_{ij}\oplus 0$ or $0\oplus E_{ij}$, depending on $k$, and set $X_k=(x_{ijk})\in \M_n(F(V))$.
We call the $F$-algebra $R(F,n)$ generated by $X_1$ and $X_2$ 
the {\it ring of (two) generic $n$-by-$n$ matrices over $F$}
(\cite[Ch.~14]{Salt:lect}).
It is a noncommutative domain
that can be specialized to give any central simple $L$-algebra of degree $n$, for every extension 
$L$ of $F$ \cite[14.1]{Salt:lect}.  The field of fractions of its center $C(F,n)\subset R(F,n)$ is 
denoted by $Z(F,n)$, and called the \emph{center of generic $n$-by-$n$ matrices over $F$}.
The (central) localization of $R(F,n)$ by the nonzero elements of $C(F,n)$ is denoted by $\UD(F,n)$,
and called the {\it generic division algebra of degree $n$ over $F$}.  
In the language of \cite[p.~11]{GMS}, the class of $\UD(F,n)$ in $\H^1(Z(F,n), \PGL_n)$ is a versal torsor.


There is another way to view the above construction, due to Procesi. 
Let $\PGL_n(F) = \GL_n(F)/F^\times$ be the projective linear group. 
This group has a representation on $V$
via the action
$A \cdot (B_1,B_2) = (AB_1A^{-1},AB_2A^{-1})$. 
It follows that $\PGL_n(F)$ acts on $F(V)$, 
and one can show the invariant field $F(V)^{\PGL_n(F)}$ is $Z(F,n)$. 
Furthermore,
$\PGL_n(F)$ acts naturally on $\M_n(F(V)) = \M_n(F) \otimes_F F(V)$ via the action on each tensor factor,
and one can show the invariant ring $\M_n(F(V))^{\PGL_n(F)}$ is $\UD(F,n)$
(see \cite[Thm 14.16]{Salt:lect} for proofs of both results).
This invariant field point of view puts the problems of this section in the bigger 
context of birational invariant fields of reductive groups and particularly the birational invariant 
fields of almost free representations of reductive groups.

One can begin with the above construction and modify it to make other generic constructions. 
For example, put $D := \UD(F,n)$ and $Z := Z(F,n)$ and assume $m$ divides $n$.  
Let $Z_m(F,n)$ be the generic splitting field of the division algebra equivalent to $D^{\otimes m}$ over $Z$, and set 
$\UD_m(F,n) =D \otimes_Z Z_m(F,n)$. 
Then $\UD_m(F,n)$ is a generic division algebra of degree $n$ and period $m$ 
with center $Z_m(F,n)$.  In the language of \cite{GMS}, the class of $\UD_m(F,n)$ in $H^1(Z_m(F,n), \GL_n/\mu_m)$ is a versal torsor.  Another example uses
Procesi's result that $Z(F,n) = F(M)^{S_n}$ where $S_n$ is the symmetric group and $M$ is a specific $S_n$ lattice
(see \cite[Thm 14.17]{Salt:lect}).
The group $S_n$ appears because it is the Galois group of the Galois closure of a ``generic'' 
maximal subfield of $\UD(F,n)$. 
It follows that if $H\subset S_n$ is a subgroup and we set $Z_H(F,n)=F(M)^H$, then $\UD_H(F,n)=D\otimes_Z Z_H(F,n)$ is a generic central simple algebra having $H$ as the Galois group of the Galois closure of a maximal subfield. In particular, if $H$ has order $n$ and acts transitively on $\{1,\ldots,n\}$ (we say $H$ is a transitive subgroup of $S_n$) we have formed the generic $H$-crossed product algebra and its center. 
One can combine these constructions and form generic crossed products 
of period $m$, but we will have nothing to say about these.

The general problem addressed here concerns the properties of the fields $Z(F,n)$ and their subsidiary fields $Z_m(F,n)$ and $Z_H(F,n)$. To state the questions, define $L/F$ to be \emph{rational} if $L$ is purely transcendental over $F$. Define $L/F$ to be \emph{stably rational} if there is a rational $L'/L$ such that $L'/F$ is rational. More generally, say that $L_1/F$ and $L_2/F$ are \emph{stably isomorphic} if there are rational $L_i'/L_i$ such that $L_1'\cong L_2'$ over $F$. Finally, define $L/F$ to be \emph{retract rational} if the following holds. There is a localized polynomial ring $R = F[x_1, \ldots, x_r](1/s)$ and an $F$ subalgebra $S \subset R$ with an $F$ algebra retraction $R \to S$ (meaning $S \to R \to S$ is the identity) such that $L$ is the field of fractions of $S$. It is pretty clear that rational implies stably rational implies retract rational, and in fact (much harder) these implications cannot be reversed.

To simplify the discussion note the result of Katsylo \cite{Katsylo} and Schofield \cite{Schofield} that if $n = ab$ for $a$ and $b$ relatively prime, then $Z(F,n)$ is stably isomorphic to the field compositum $Z(F,a)\,Z(F,b)$. Similar statements are possible for the $Z_m(F,n)$ and $Z_H(F,n)$. This often (but not always) allows reduction to the case where $n$ is a prime power.

\begin{prob} \emph{\cite[p.~240]{procesi:67}} \label{center.prob}
Is $Z(F,n)/F$ rational, stably rational, or retract rational? The same question for $Z_m(F,n)$ and $Z_H(F,n)$.
\end{prob}

This appears as Problem 8 in \cite{Salt:fd} (see also Saltman's Problems 9, 10, and 11), and is a major topic of the surveys \cite{Lebruyn:surv} and \cite{Formanek:surv}.
We note that $Z(F,n)/F$ is retract rational if and only if division algebras of degree $n$ have the so-called lifting  property, see \cite[p.~77]{Salt:lect}. Similar statements can be made for $Z_m(F,n)$ and $Z_H(F,n)$.

When $n$ is 2, 3, or 4 then $Z(F,n)$ is rational, as proved (respectively) by Sylvester, Procesi and Formanek.  When $n$ is 5 or 7 Bessenrodt and Lebruyn showed in \cite{BlB} that $Z(F,n)$ is stably rational, and a second, more elementary proof of this was given by Beneish \cite{Beneish}. 

There are a few results for the $Z_m(F,n)$. The field $Z_2(F,4)$ is stably rational by Saltman \cite{Salt:symp}, and Saltman--Tignol \cite{SaltTig} showed that $Z_2(F,8)$ is retract rational. Beneish \cite{Beneish:28} showed that $Z_2(F,8)$ is stably rational. 
Finally, we leave to the well-read reader to prove that when $H \subset S_n$ is cyclic and transitive and $F$ contains a primitive $n$-th root of unity, then $Z_H(F,n)$ is rational (and results \emph{are} available for general $F$, $n$ not of the form $8m$, and retract rationality). Thus the first interesting specific question along these lines is:
\begin{prob}
Determine if $Z_H(F,9)$ is rational, stably rational, or retract rational, where 
 $H = C_3 \times C_3 \subset S_9$ is transitive. 
 \end{prob}

\section{Essential dimension}\label{ed.sec}

\newcommand{\func}{\mathscr{F}}

Essential dimension counts the number of parameters needed to define
an algebraic structure.  This notion was introduced in the late 1990s
by Buhler--Reichstein \cite{buhler_reichstein:ed} and placed in a
general functorial context by Merkurjev
\cite{berhuy_favi:ed_functorial}.  Given a field $F$, a functor
$\func : \Fields_F \to \Sets$ from the category of field
extensions of $F$ (together with $F$-embeddings) to the category of
sets, a field extension $F \to K$, and an element $a \in
\func(K)$, a \emph{field of definition} of $a$ is a field
extension $F \to L$ and an $F$-embedding $L \to K$ such that $a$ is in
the image of the map $\func(L) \to \func(K)$.  The
\emph{essential dimension of $a$} (over $F$), denoted by $\ed_F(a)$,
is the infimum of the transcendence degrees $\trdeg_F(L)$ over all
fields of definition $L$ of $a$.  Finally, the \emph{essential
dimension of $\func$} is the number
$$
\ed_F(\func)=\sup\{\ed_F(a)\}
$$
where the supremum is taken over all field extensions
$F \to K$ and all elements $a \in \func(K)$.
We will suppress the dependence on the base field $F$ when
no confusion may arise.

The \emph{essential $p$-dimension} of $\func$, denoted by $\ed_{F,p}(\func)$ or simply $\ed_p(\func)$, is defined similarly.  
One replaces $\ed_F(a)$ with $\ed_{F,p}(a)$, which is the infimum of $\ed_F(a_{K'})$ as $K'$ varies over all finite embeddings $K \ra K'$ of prime-to-$p$ degree.  Obviously, $\ed_{F,p}(a) \le \ed_F(a)$ for all $a$, hence $\ed_{F,p}(\func) \le \ed_F(\func)$.

\begin{defn}
The \emph{essential dimension} $\ed(G)=\ed_F(G)$ of an algebraic
group $G$ over $F$ is the essential dimension of the
functor $\H^1(-,G)$.  Similarly, the \emph{essential $p$-dimension}
$\ed_p(G) = \ed_{F,p}(G)$ of $G$ is the essential
$p$-dimension of the functor $\H^1(-,G)$.
\end{defn}

The essential dimension of a algebraic group $G$ equals the essential
dimension of any
versal $G$-torsor, see \cite[\S6]{berhuy_favi:ed_functorial}.  For a
broad survey of essential dimension results, see Reichstein's
sectional address \cite{R:ICM} to the 2010 ICM in Hyderabad, India.

The functor $\H^1(-,\PGL_n)$ is identified via Galois descent with the
functor assigning to a field extension $F \to K$ the isomorphism
classes of central simple $K$-algebras of degree $n$.  
As we shall discuss below, computation of
the essential dimension of $\PGL_n$ is relevant to the theory of
central simple algebras, in the sense that knowing the ``number of
parameters'' needed to define central simple algebras of degree $n$
can lead to conclusions about their structural properties (e.g.\
concerning decomposability and crossed product structure).

\begin{prob}
\label{ed:mainprob}
Compute the essential dimension and the essential $p$-dimension of
$\PGL_{n}$.
In particular:
\begin{enumerate}
\item Compute $\ed_F(\PGL_4)$ when $\chr F=2$.
\item Compute $\ed_F(\PGL_5)$.
\end{enumerate}
\end{prob}

\Wik The generic division algebra $\UD(F,n)$ of degree $n$ corresponds
to a versal $\PGL_n$-torsor, see \S\ref{center.sec}.  It is immediate
from Procesi's description of $Z(F,n)$ by means of invariants that
$\ed(\PGL_n) \leq n^2+1$, see \cite[Th.\ 1.8]{procesi:67}.  With some
additional work, Procesi \cite[Th.\ 2.1]{procesi:67} proved that
$\ed(\PGL_n) \leq n^2$.  
The current best upper bounds are
$$
\ed(\PGL_n) \leq
\begin{cases}
\frac{1}{2}(n-1)(n-2) & \text{if $n \geq 5$ and $n$ is odd} \\
n^2 - 3n + 1 & \text{if $n \geq 4$, $F=F_{\sep}$, and $\mathrm{char}(F)=0$}
\end{cases}
$$
see \cite[Th.\ 1.1]{LRRS}, \cite[Prop.\ 1.6]{Le:ed}, and
\cite{FF:edt}.  

Tsen's theorem can be used to show that $\ed(\PGL_n) \geq 2$ for any
$n \geq 2$, see \cite[Lemma 9.4a]{R:notion}.  The current best lower
bounds are 
$$
\ed(\PGL_{n}) \geq
\ed_p(\PGL_{n}) \geq (r-1)p^{r}+1
$$
where $p^r$ is the highest power of $p$ dividing $n$ and we assume $\chr F \ne p$, see Merkurjev
\cite{M:pn} and the discussion below.  This improves on the
long-standing lower bound $\ed(\PGL_{p^r}) \geq \ed_p(\PGL_{p^r}) \geq
2r$ in \cite[Th.\ 16.1]{R:hermite}, \cite[Th.\
8.6]{reichstein_youssin}.  Finally, because of the decomposition of
central simple algebras into prime powers, $\ed(\PGL_{nm}) \leq
\ed(\PGL_n) + \ed(\PGL_m)$ if $(n,m)=1$.

Table \ref{table:ed} lists current bounds for $\ed(\PGL_n)$ for small
values of $n$, obtained by combining the bounds in this discussion
with those after the discussion of Problem \ref{prob:ed_eric}.  (A slightly different table on the same topic can be found at the end of Baek's thesis \cite{Baek:th}.)
Of course, the bounds may improve upon further specification of the
field.  The cases for $n=4$ in characteristic $2$ and for $n=5$ are
the smallest unknown cases, and each has important implications for
the crossed product problems of \S\ref{Q2} and \S\ref{Q1},
respectively (see discussion below).  The feeble bounds on the second line of the table---when the characteristic divides $n$---illustrate our lack of knowledge in this case.  One of the few positive results is that the upper bound $\ed(\PGL_4) \le 5$ from \cite[Cor.~3.10(a)]{LRRS} holds without hypotheses on the base field.

\setcounter{table}{\value{section}}
\addtocounter{table}{-1}
\begin{table}[h!]
\begin{center}
\begin{tabular}{|c||c|c|c|c|c|c|c|c|c|}  \hline
$n$          & 2 & 3 & 4 & 5     & 6     & 7      & 8      &
9&10\\  \hline\hline
\parbox{1.75cm}{$\ed(\PGL_n)$ \\ $\chr F \notdivides n$}& 2 & 2
& \parbox{.9cm}{\quad 5} & \parbox{.85cm}{2$^{\dagger}$ -- 6 \\ 3
-- 6} & \parbox{.3cm}{\hfill 2$^{\dagger}$ ~\hfill \\ 3} &
\parbox{1.1cm}{2$^{\dagger}$ -- 15\\3 -- 15} & 17 -- 64& 10 -- 28&
\parbox{.85cm}{2$^{\dagger}$ -- 8\\3 -- 8}\\[.7em] \hline
\parbox{1.75cm}{$\ed(\PGL_n)$\\ $\chr F \divides n$}& 2 & 2
& 2 -- 5 & 2 -- 6 & 2 -- 3 & 2 -- 15 & 2 -- 64& 2 -- 28&2 -- 8\\[.7em] \hline
\end{tabular}
\end{center}
\caption{Bounds for $\ed(\PGL_n)$, references found in the
text. {}$^{\dagger}$ means ``when $F$ contains $\zeta_{n/2}$ if $n$ is even or
$\zeta_n+\zeta_n^{-1}$ if $n$ is odd''}
\label{table:ed}
\end{table}

Much more is known about essential $p$-dimension.  By
Reichstein--Youssin \cite[Lemma 8.5.5]{reichstein_youssin},
$\ed_p(\PGL_n) = \ed_p(\PGL_{p^r})$ if $p^r$ is the highest power of
$p$ dividing $n$; in particular, $\ed_p(\PGL_n) = 0$ if $p$ does not
divide $n$.  Every central simple algebra of degree $p$ becomes cyclic
over a prime-to-$p$ extension, eventually leading to $\ed_p(\PGL_p) =
2$, see \cite[Lemma 8.5.7]{reichstein_youssin}. Using an extension of
Karpenko's incompressibility theorem to products of $p$-primary
Severi--Brauer varieties, Karpenko--Merkurjev
\cite{KM:edp} provide a formula for the
essential dimension of any finite $p$-group, considered over a field
containing the $p$-th roots of unity.  This general formula was
extended to twisted $p$-groups and algebraic tori in
\cite{lotscher_macdonald_meyer_reichstein:ed_tori}, and ultimately
used with great success by Merkurjev \cite{M:p2}, \cite{M:pn} to
establish the formula $\ed_p(\PGL_{p^2}) = p^2+1$ 
for any field $F$ with $\chr F \neq p$, 
and the current best lower bound
\begin{equation} \label{M.lb}
(r-1)p^r + 1 \leq \ed_p(\PGL_{p^r}) \leq p^{2r-2} + 1.
\end{equation}
The above upper bound is in a recent preprint by Ruozzi \cite{R:ed},
improving on \cite[Th.\ 1.1]{MR:ub}.  For
$p=2$ and $r=3$, note that the upper and lower bounds coincide,
yielding $\ed_2(\PGL_8) = 17$ when $\chr F\neq 2$.

\subsection*{Asymptotic bounds}
It might be illustrative to view these bounds on $\ed(\PGL_n)$ asymptotically, in terms of big-$O$ notation.  In that language, we have the naive upper bound that $\ed(\PGL_n)$ is $O(n^2)$ and the naive lower bound that it is $\Omega(1)$, because it is between $n^2 + 1$ and 2.  The furious profusion of bounds listed in this section are invisible in this context, except for Merkurjev's lower bound \eqref{M.lb}, which shows that $\ed(\PGL_n)$ is \emph{not} $O(n)$.  (This settled in the negative Bruno Kahn's ``barbecue problem'' posed in 1992 \cite[\S1]{MR:ub}.)  The gap between the upper and lower bounds for  $\PGL_n$ stands in interesting constrast with the situtation for $\Spin_n$: $\ed(\Spin_n)$ is both $O(\sqrt{2}^n)$ and $\Omega(\sqrt{2}^n)$, i.e., is asymptotically bounded both above and below by constants times $\sqrt{2}^n$, by \cite{BRV:spin}.

\subsection*{Essential dimension and crossed products}\label{edcp}

Let $G$ be a finite group of order $n \geq 2$.  Let $\CP_n$ (resp.\
$\Alg_G$) be the functor $\Fields_F \to \Sets$ assigning to $F \to K$
the set of isomorphism classes of crossed product (resp.\ $G$-crossed
product) $K$-algebras of degree $n$. Of course, $\ed(\CP_n)$ is the
maximum of $\ed(\Alg_G)$ over all groups $G$ of order $n$. These
functors are relevant to the crossed product problem discussed in \S\ref{Q2}.
For example, $\ed(\Alg_G) \leq \ed(\CP_n) \leq \ed(\PGL_n)$, with
equality if every central simple algebra of degree $n$ over every
field extension of $F$ is a $G$-crossed product. An inequality
$\ed(\CP_n) < \ed(\PGL_n)$ would imply the existence of a noncrossed
product algebra of degree $n$ over some field extension of $F$.

\begin{prob}
\label{prob:ed_eric}
Determine structural conclusions about central simple algebras 
from essential ($p$-)dimension bounds.
\begin{enumerate}
\item Calculate bounds for $\ed(\CP_n)$ and $\ed(\Alg_G)$.  
\item Compute the difference $\ed(\Alg_G)-\ed(G)$.
\item If $\ed(\CP_n)=\ed(\PGL_n)$, determine whether every algebra of
degree $n$ over every extension of $F$ is a crossed product.
\end{enumerate}
\end{prob}

\Wik By \cite[IX.6, Theorem 9]{Albert}, $\H^1(-,\PGL_n)$ is isomorphic
to the functor $\Alg_G$ for $n=2,3,6$ and $G=C_n$, and for $n=4$ and
$G=C_2\times C_2$.
A result communicated to us by Merkurjev states that
$\ed(\Alg_{C_n})=\ed(C_n)+1$ when $\chr F$ does not divide $n$ (see
Prop.~\ref{thm:merkcyclic}). 
Since $\ed(C_2)=\ed(C_3)=1$ (over any
field, see cf.\ \cite[\S1.1]{serre:topics}, \cite[\S2.1]{JLY:generic}),
and $\ed(C_6)$ equals $1$ or $2$ (depending on whether $F$ contains
the $6$-th roots of unity or not), we obtain
the values listed in Table \ref{table:ed} for $n=2,3,6$ and $\chr F$ not dividing $n$. 
When $n=6$ and $\chr F$ divides $n$, one can show that $\ed(C_6) = 2$
using \cite[Th.\ 1]{L:ed1}, and therefore $2\leq \ed(\PGL_6) \leq 3$.
Note that Prop.~\ref{thm:merkcyclic} resolves
Problem~\ref{prob:ed_eric}(2) for $G=C_n$, when $\chr F$ does not divide $n$.

In \cite[Th.\ 7.1]{M:pn}, Merkurjev proved that $\ed(\Alg_{C_2\times
C_2})=5$ when $\chr F \neq 2$, and thus $\ed(\PGL_4)=5$ (see also Rost
\cite{rost:edPGL_4}).
Note that since $\ed(\Alg_{C_4}) = \ed(C_4)+1 \leq 3$ when $\chr F
\neq 2$, this implies the existence of noncyclic algebras of degree 4
(and period 4), as exhibited in \cite{Albert:deg4eg}.

One can bound $\ed(\Alg_G)$ by bounding the number of generators of
$G$.  If a finite group $G$ of order $n$ can be generated be $r \geq
2$ elements, then 
$\ed(\Alg_G) \leq (r-1)n + 1$, see \cite[Cor.\ 3.10a]{LRRS}.  This shows $\ed(\PGL_4)\leq 5$ when $\chr F=2$.  The
bound is sharp for $\chr F\neq p$: if $G = C_p^r$ for $r\geq 2$ and $\chr F\neq p$,
then $\ed(\Alg_G) = \ed_p(\Alg_G) = (r-1)p^r + 1$ by \cite[Th.\
7.1]{M:pn}.  Note that this resolves Problem \ref{prob:ed_eric}(1) for
elementary abelian $G$.
One can bound the number of generators by $r \leq \log_2(n)$ for $n
\geq 4$, see \cite[Cor.\ 3.10b]{LRRS}, and this bound is realized on
elementary abelian 2-groups.  Thus
$\ed(\CP_n) \leq (\log_2(n)-1)n+1$ for $n\geq 4$.  Now we have:
\[
\ed(\CP_8) \le 17 = \ed_2(\PGL_8) \le \ed(\PGL_8).
\]
As there exist non-crossed products of degree 8, it is reasonable to guess that at least one of the two inequalities is strict.


\newcommand{\Amn}[2]{\GL_{#1}/\mu_{#2}}
\subsection*{Algebras with small exponent}
The functor $\H^1(-, \GL_n/\mu_m)$ assigns to a field $K$ the isomorphism classes of central simple $K$-algebras of degree
$n$ and period dividing $m$.
Of course, $\H^1(-, \GL_n/\mu_n) = \H^1(-,\PGL_n)$.  By
Albert's theorem, every algebra of degree 4 and period 2 is
biquaternion, ultimately yielding $\ed(\Amn{4}{2}) =
\ed_2(\Amn{4}{2})=4$ when $\chr F\neq 2$, see \cite[Rem.\
8.2]{BM:algnm}.   
When $\chr F=2$, all algebras of degree 4 and period 2 are cyclic by
\cite{A:deg4char2}, hence $\ed(\Amn{4}{2})=\ed(\Alg_{C_4}) \leq
\ed(C_4) + 1 = 3$ (one can bound $\ed(C_4)$ using Witt vectors of
length 2, see \cite[Th.\ 8.4.1]{JLY:generic}).
Similarly, for $\Amn{8}{2}$, Rowen \cite{Rowen:CSA} produced
triquadratic splitting fields, ultimately yielding $\ed(\Amn{8}{2}) =
\ed_2(\Amn{8}{2}) = 8$ when $\chr F \neq 2$, see \cite[Cor.\
8.3]{BM:algnm}.  In particular, this implies the existence of algebras
of degree 8 and period 2 (when $\chr F\neq 2$) that are not
decomposable as a product of three quaternion algebras; 
such examples were exhibited in \cite{ART}.  

When $\chr F \neq p$, the
current best bounds for $\ed_p(\Amn{p^r}{p^s})$ are
$$
\left.\begin{array}{l}
(r-1)2^{r-1} \\
(r-1)p^r + p^{r-s} 
\end{array}\right\}
\leq \ed_{p}(\Amn{p^r}{p^s}) \leq
\begin{cases}
2^{2r-2} & \text{for $p=2$ and $s=1$} \\
p^{2r-2} + p^{r-s} & \text{otherwise}
\end{cases}
$$
for $r \geq 2$ and $1 \leq s \leq r$, see \cite[Th.\ 6.1, 7.2]{BM:algnm}, \cite[Th.\ 1.1]{R:ed}, and
\cite[Th.\ 4.1.30]{Baek:th}.   For $p$ odd and $r = 2$, the upper and lower bounds coincide, yielding
\[
\ed_p(\Amn{p^2}{p^2}) = p^2+1 \quad \text{and} \quad \ed_p(\Amn{p^2}{p}) = p^2 + p.
\]
This implies the existence of indecomposable algebras of degree $p^2$ and
period $p$, as exhibited previously in \cite{Tig:cor}, and
is an example of an essential dimension result having implications
on the existence of algebras with certain structural properties.

The smallest open cases seem to be:

\begin{prob}
Compute $\ed_F(\Amn{8}{2})$ when $\chr F = 2$; $\ed_F(\Amn{16}{2})$ when $\chr F \ne 2$; and $\ed_F(\Amn{p^2}{p})$ for odd primes $p$ and all $F$.
\end{prob}

The last part overlaps with Problem \ref{Q2.prob2}.

We prove the essential dimension result stated above:

\begin{prop} \label{thm:merkcyclic}
If $\chr F$ does not divide $n>1$, then
$\ed_F(\Alg_{C_n})=\ed_F(C_n)+1$.
\end{prop}

\begin{proof} 
Since clearly $\ed_F(\Alg_{C_n})\leq\ed_F(C_n)+1$, it
suffices to produce an algebra with essential dimension at least
$\ed_F(C_n)+1$. Suppose $K/F$ is a field extension, $L/K$ is a cyclic
field extension of degree $n$, and $A=(L/K,\sigma,t)$ is a cyclic
crossed product over the rational function field $K(t)$. The parameter
$t$ defines a discrete valuation $v$ on $K$, and we have a residue
$\partial_v(A)=\chi$, for an element $\chi$ of order $n$ in the
character group $X(K)$. If $A$ is defined as $A_0$ over an extension
$E/F$, then $t$ defines a discrete valuation $v_0$ on $E$, and
$\partial_{v_0}(A_0)=\chi_0\in X(\kappa(v_0))$. Since the residue map
commutes with scalar extension---scaled by the ramification index if necessary---$\chi_0$ has order at least $n$, hence
$\ed_F(C_n)\leq\trdeg_F(\kappa(v_0))$. Since
$\trdeg_F(\kappa(v_0))<\trdeg_F(E)$,
we conclude $\ed_F(C_n)<\ed_F(A)$. 
\end{proof}


\section{Springer problem for involutions}

Involutions on central simple algebras are ring-antiautomorphisms of
period~$2$. They come in different types, depending on their action on
the center and on the type of bilinear forms they are adjoint to after
scalar extension to an algebraic closure of the center: an involution
is of \emph{orthogonal} (resp.\ \emph{symplectic}) \emph{type} if it
is adjoint to a symmetric, nonalternating bilinear form (resp.\ to an
alternating bilinear form) over an algebraic closure; it is of
\emph{unitary type} if its restriction to the center is not the
identity. The correspondence between central simple algebras with
involution and linear algebraic groups of classical type was first
pointed out by Andr\'e Weil \cite{Weil}; it is a deep source of
inspiration for the development of the theory of central simple
algebras with involution, see \cite{KMRT}. For semisimple linear
algebraic groups, Tits defined in \cite{Ti:Cl} a notion of index, which
generalizes the Schur index of central simple algebras and
the Witt index of quadratic forms. Anisotropic general linear groups
arise from division algebras, and anisotropic quadratic forms yield
anisotropic orthogonal groups. More generally, every adjoint group of
type~$^1\!D_n$ or $^2\!D_n$ over a field of characteristic different
from~$2$ can be
represented as the group of automorphisms of a central simple algebra
$A$ of degree~$2n$ with orthogonal involution $\sigma$; the group is
anisotropic if and only if $\sigma$ is anisotropic in the following
sense: for $x\in A$, the equation $\sigma(x)x=0$ implies $x=0$. The
behavior of the index under scalar extension is a major subject of
study, to which the following problem, first formulated in
\cite[p.~475]{BST}, pertains:

\begin{prob}\label{Inv1}
Suppose $\sigma$ is an anisotropic orthogonal or symplectic
involution on a central simple algebra $A$ over a field $F$. Does
$\sigma$ remain anisotropic after scalar extension to any odd-degree
field extension of $F$?
\end{prob}

(If $\charac F=2$, the question makes sense, but to preserve the relation with linear algebraic groups one should replace the orthogonal involution $\sigma$ with a quadratic pair \cite[\S5B]{KMRT}.)

Representing $A$ as $\operatorname{End}_DV$ for some vector space $V$
over a division algebra $D$, we may rephrase the problem in terms of
hermitian forms: if a hermitian or skew-hermitian form over a division
algebra with orthogonal or symplectic involution is anisotropic, does
it remain anisotropic after any odd-degree extension? When the central
simple algebra is split (which is always the case if its degree $\deg
A$ is odd), orthogonal involutions are adjoint to symmetric bilinear
forms, and the problem has an affirmative solution by a well-known
theorem of Springer \cite[Cor.~18.5]{EKM} (actually first proven---but
not published---by E.~Artin in 1937, see \cite[Rem.~1.5.3]{Kahn:book}). On
the other hand, every
symplectic involution on a split algebra is hyperbolic, hence
isotropic, so the problem does not arise for symplectic involutions on
split algebras. At the other extreme, if $A$ is a division algebra,
then the solution is obviously affirmative, since $A$ remains a
division algebra after any odd-degree extension and every involution
on a division algebra is anisotropic.

Variants of Problem~\ref{Inv1} take into account the ``size'' of the isotropy:
call a right ideal $I\subset A$ \emph{isotropic} for an involution
$\sigma$ if $\sigma(I)\cdot I=\{0\}$, and define its \emph{reduced
dimension} by $\rdim I=\frac{\dim I}{\deg A}$. The reduced dimension
of a right ideal can be any multiple of the Schur index $\ind A$
between $\ind A$ and $\deg A$, but for isotropic ideals we have $\rdim
I\leq\frac12\deg A$, see \cite[\S6A]{KMRT}. The involution $\sigma$ is
called \emph{metabolic} if $A$ contains an isotropic ideal of reduced
dimension $\frac12\deg A$. A weak version of Springer's theorem holds
for arbitrary involutions: Bayer--Lenstra proved in
\cite[Prop.~1.2]{BayerLenstra} that an
involution that is not metabolic cannot become metabolic over an
odd-degree field extension. As a result, Problem~\ref{Inv1} also has an
affirmative solution when $\ind A=\frac12\deg A$, for in this case
isotropy implies metabolicity. (See \cite[Th.~1.14]{BFT} for the
analogue for quadratic pairs.)

If $\charac F\neq2$ and $\ind A=2$, Problem~\ref{Inv1} was solved in
the affirmative by
Parimala--Sridharan--Suresh \cite{PSS:herm}. The symplectic case is easily
reduced to the case of quadratic forms by an observation of Jacobson
relating hermitian forms over quaternion algebras to quadratic
forms. The orthogonal case also is reduced to the case of quadratic
forms, using scalar extension to the function field of the conic that
splits $A$. That approach relates Problem~\ref{Inv1} to another important
question, to which Parimala--Sridharan--Suresh gave a positive
solution when $\ind A=2$ and $\charac F\neq2$:

\begin{prob} \label{Inv2}
Suppose $\sigma$ is an anisotropic orthogonal involution on a
central simple algebra $A$ over a field $F$. Does $\sigma$ remain
anisotropic after scalar extension to the function field $F_A$ of the
Severi--Brauer variety of $A$?
\end{prob}

(This is a special case of a general problem concerning semisimple algebraic groups, namely whether a projective homogeneous variety under one group has a rational point over the function field of a projective homogeneous variety under another group.  From this perspective, Problem \ref{Inv2} is asking for an analogue of the index reduction results from \cite{SvdB}, \cite{MPW1}, \cite{MPW2}, etc.)

By Springer's theorem, an affirmative answer to Problem \ref{Inv2} readily
implies that Problem~\ref{Inv1} has an affirmative solution for orthogonal
involutions. Surprisingly, Karpenko \cite{Karp:odd} recently proved that the
converse also holds when $\charac F\neq2$: an
orthogonal involution cannot become isotropic over $F_A$ unless it
also becomes isotropic over some odd-degree extension of $F$.

Problem~\ref{Inv2} is addressed in the following papers
(besides~\cite{PSS:herm} and
\cite{Karp:odd}): \cite{Karp:herm}, \cite{Karp:iso}, \cite{G:16},
\cite{Karp:hyp}. In
\cite{Karp:herm} and \cite[2.5]{Karp:ICM}, Karpenko gives an
affirmative solution when $A$ is a
division algebra. This result is superseded in
\cite{Karp:iso}, where he proves that for any quadratic pair on a
central simple algebra $A$ of arbitrary characteristic, the Witt index
of the quadratic form to which the quadratic pair is adjoint over
$F_A$ is a multiple of $\ind A$. In particular, it follows that if the
quadratic pair becomes hyperbolic over $F_A$, then $\ind A$ divides
$\frac12\deg A$. An alternative proof of the latter result is given by
Zainoulline in \cite[Appendix~A]{G:16}. A much stronger statement was
soon proved by Karpenko \cite{Karp:hyp}: if $\charac F\neq2$, an
orthogonal involution that is not hyperbolic cannot become hyperbolic
over $F_A$. (The special case where $\deg A=8$ and $\ind A=4$ was
obtained earlier by Sivatski in~\cite[Prop.~3]{Sivatski:app}, using Laghribi's
work on $8$-dimensional quadratic forms \cite[Th\'eor\`eme~4]{Lag:Duke}.)

Note that the orthogonal and symplectic cases of Problem~\ref{Inv1} are related
by the following observation: tensoring a given central simple
$F$-algebra with involution $(A,\sigma)$ with the ``generic''
quaternion algebra $(x,y)_{F(x,y)}$ with its conjugation involution
yields a central simple $F(x,y)$-algebra with involution
$(A',\sigma')$, which is anisotropic if and only if $\sigma$ is
anisotropic. If $\charac F\neq2$, the involution $\sigma'$ is
orthogonal (resp.~symplectic) if and only if $\sigma$ is symplectic
(resp.~orthogonal). Therefore, a negative solution to Problem~\ref{Inv1} for a
given degree~$d$ and index~$i$ in the orthogonal (resp., symplectic)
case readily yields a negative solution in the symplectic (resp.~orthogonal) case in degree~$2d$ and index~$2i$. The same idea can be
used to obtain symplectic versions of the results on Problem~\ref{Inv2}, where
$F_A$ is replaced by a field over which the index of $A$ is
generically reduced to~$2$, see \cite[App.~A]{Karp:hyp}.  

If $\charac F\neq2$, the smallest index for which Problem~\ref{Inv1} is open
is~$4$, and the smallest degree is~$12$. Since
Parimala--Sridharan--Suresh do not address the characteristic~$2$ case
in \cite{PSS:herm}, Problem~\ref{Inv1}---restated for quadratic pairs---seems to be open in this case already for
index~$2$ and degree~$8$.  (The answer to Problem \ref{Inv1} is ``yes'' in the degree 6 case.  This can be seen via the exceptional isomorphism $D_3 = A_3$.)

\subsection*{Unitary involutions}
The analogue of Problem~\ref{Inv1} for unitary involutions was solved in the
negative in \cite{PSS:herm}: for any odd prime~$p$, there is a central
simple algebra $A$ of degree~$2p$ and index~$p$ with anisotropic
unitary involution over a field $F=\ell((t))$, where $\ell$ is a
ramified quadratic extension of a $p$-adic field $k$, which becomes
isotropic over any odd-degree extension of $F$ that splits $A$. If
$p\geq5$, there are unitary involutions on division algebras of
degree~$p$ that become isotropic after scalar extension to any maximal
subfield. As Parimala--Sridharan--Suresh suggest at the end of their
paper, the correct analogue of Problem~\ref{Inv1} for unitary involutions
should probably ask: Does the involution remain anisotropic after
scalar extension to a field extension of degree coprime to $2\ind A$?

Motivated by this observation, we can generalize Problem \ref{Inv1} by asking:
\begin{prob} \label{Inv3}
Let $G$ be an absolutely almost simple linear algebraic group that is
anisotropic over $F$.  Does $G$ remain anisotropic over every finite
extension $K/F$ of dimension not divisible by any prime in $S(G)$?
\end{prob}
Here $S(G)$ denotes the set of `homological torsion primes' from
\cite[2.2.3]{SeCGp}, exhibited in Table \ref{HTP}.
Roughly speaking, the results on Problem \ref{Inv1} for orthogonal and
symplectic involutions concern the types $B$, $C$, and $D$ cases of
Problem \ref{Inv3} and the results for unitary involutions concern the
type $A$ case.  The answer to Problem \ref{Inv3} is ``yes'' for $G$ of
type $G_2$ or $F_4$; this can be seen by inspecting the cohomological
invariants of these groups described in \cite{KMRT}.
\setcounter{table}{\value{section}}
\addtocounter{table}{-1}
\begin{table}[tbh]
\begin{tabular}{c|c}
type of $G$&elements of $S(G)$\\ \hline
$A_n$&2 and prime divisors of $n+1$\\
$B_n$, $C_n$, $D_n$ ($n > 4$), $G_2$&2\\
$D_4$, $F_4$, $E_6$, $E_7$&2 and 3\\
$E_8$&2, 3, and 5
\end{tabular}
\caption{The set $S(G)$ of homological torsion primes for an absolutely almost simple algebraic group $G$} \label{HTP}
\end{table}


\section{Artin--Tate Conjecture}

Motivated by an analogue of the Birch and Swinnerton-Dyer conjecture
for abelian varieties over global fields of characteristic $p$, Tate
(reporting on joint work with Artin) \cite{Ta66} introduces a host of
conjectures for surfaces over finite fields. As most of the progress
on these conjectures has been made from their connection with the Tate
conjecture on algebraic cycles and the Birch and Swinnerton-Dyer
conjecture, a natural challenge arises: \emph{Can progress be made on
the conjectures in this section using ``central simple algebra
techniques''?}

Let $k=\F_q$ be a finite field of characteristic $p$ and $X$ a
geometrically connected, smooth, and projective $k$-surface.  Let
$\Br(X) = H_{\text{\'et}}^2(X,\Gm)$, which equals the Azumaya Brauer group since $X$ is regular over a field.

\begin{conj*}[Artin--Tate Conjecture A]
The Brauer group of $X$ is finite.
\end{conj*}

Since $\Br(X) = \Br_{ur}(k(X))$, where $\Br_{ur}(k(X))$ denotes the
unramified (at all codimensional 1 points) Brauer group of the
function field $k(X)$ of $X$, Conjecture A is equivalent to:

\begin{conj*}[Artin--Tate Conjecture A]
If $k(X)$ be a function field of transcendence degree 2 over a finite
field, then there are finitely many unramified classes in $\Br(k(X))$.
\end{conj*}

\subsection*{Relationship with the Tate Conjecture on algebraic cycles}

Let $\ell\neq p$ be a prime.  Let $P_2(X,t)$ be the characteristic
polynomial of the action of the (geometric) Frobenius morphism on the
$\ell$-adic cohomology $\H^2(\ol{X},\Q_{\ell}(1))$, where $\ol{X} = X
\times_{k} k_{\sep}$.  Let $\NS(X) = \Pic(X)/\Pic^0(X)$ denote the
\emph{N\'eron--Severi group} of $X$ and $\rho(X) = \rk{\Z} \NS(X)$ its
\emph{Picard number}.  Tate \cite[\S3]{Ta65} has earlier conjectured
a relationship between the Picard number and this characteristic
polynomial for surfaces.

\begin{conj*}[Tate's Conjecture T]
The Picard number of $X$ is equal to the multiplicity of the root
$t=q^{-1}$ in the polynomial $P_2(X,t)$.
\end{conj*}

Tate \cite[\S4, Conjecture C]{Ta66} refines this by providing a
conjectural leading term for the polynomial $P_2(X,t)$ at $t=q^{-1}$.

\begin{conj*}[Artin--Tate Conjecture C]
$$
\lim_{s \to 1} \frac{P_2(X,q^{-s})}{\bigl(1-q^{1-s}\bigr)^{\rho(X)}}
= \frac{|\!\Br(X)|\, |\det(h_{\NS(X)})|}{q^{\alpha(X)} \,
|\NS(X)_{\mathrm{tors}}|^2}
$$
where $h_{\NS(X)}$ is the intersection form on the N\'eron--Severi
group modulo torsion, $\alpha(X) = \chi(X,\OO_X) - 1 + \dim\PPic_X$,
and $\PPic_X$ is the Picard variety of $X$.
\end{conj*}

In fact, it turns out that showing the finiteness of $\Br(X)$ is the
``hard part,'' and all of the above conjectures are equivalent.

\begin{thm} \label{thm:milne}
The following statements are equivalent:
\begin{enumerate}
\item Conjecture A holds for $X$, i.e.\ $\Br(X)$ is finite.

\item $\Brp[\ell](X)$ is finite for some prime $\ell$ (with $\ell=p$
allowed).

\item Conjecture T holds for $X$.

\item Conjecture C holds for $X$.
\end{enumerate}
\end{thm}

In \cite[Th.~5.2]{Ta66}, the prime-to-$p$ part of this theorem is
proved, i.e.\ the finiteness of the prime-to-$p$ part of the Brauer
group $\Brp[p'](X)$ is equivalent to Conjecture C up to a power of
$p$.  Milne \cite[Th.~4.1]{Mi75} proves the general case using
comparisons between \'etale and crystalline cohomology.  Note that in
\cite{Mi75}, Milne assumes that $p \neq 2$, though on his website
addendum page, he points out that by appealing to Illusie \cite{IlCot}
in place of a preprint of Bloch, this hypothesis may be removed.

In turn, Conjecture T is a special case of the Tate Conjecture on
algebraic cycles \cite[Conjecture 1]{Ta65}, made at the AMS Summer
Institute at Woods Hole, 1964.  Denote by $Z^i(X)$ the group of
algebraic cycles of codimension $i$ on a variety $X$.

\begin{conj*}[Tate Conjecture]
Let $k$ be a field finitely generated over its prime field,
$\Gamma = \Gal(k_{\sep}/k)$, and $X$ a geometrically connected, smooth,
and projective $k$-variety.  Then the image of the cycle map
$$
\mathrm{cl} : Z^i(\ol{X}) \to \H^{2i}(\ol{X},\Q_{\ell}(i))
$$
spans the $\Q_{\ell}$-vector subspace
$\H^{2i}(\ol{X},\Q_{\ell}(i))^{\Gamma}$ of Galois invariant classes.
\end{conj*}

In the case of interest to us, of (co)dimension one cycles on a
surface over a finite field, the cycle map factors through the
N\'eron--Severi group
and Milne \cite[Th.~4.1]{Mi75} (following Tate \cite[\S3]{Ta65})
proves that the Tate conjecture (in this case, equivalent to $\rho(X)
= \rk{\Z} \H^2(\ol{X},\Q_{\ell}(1))^{\Gamma}$) is equivalent to
Conjecture T.

\Wik Conjecture T is known for many special classes of surfaces:
rational surfaces (Milne \cite{Mi70}),
Fermat hypersurfaces $X_0^n + X_1^n + X_2^n + X_3^n = 0$, with
$p \equiv 1 \bmod n$ or $p^a \equiv -1 \bmod n$ for some $a$
(Tate \cite[\S3]{Ta65}),
abelian surfaces and products of two curves (Tate \cite[Th.~4]{Ta66a}),
$K3$ surfaces with a pencil of elliptic curves (Artin and
Swinnerton-Dyer \cite[Th.~5.2]{ArSD73}),
$K3$ surface of finite height (where $p \geq 5$) (Nygaard
\cite[Cor.~3.4]{Ny83} and Nygaard--Ogus \cite[Theorem
0.2]{NyOg85}),
and
certain classes of Hilbert modular surfaces
(Harder--Langlands--Rapoport \cite{HLR86}, Murty--Kumar--Ramakrishnan
\cite{MKR87}, Langer \cite{Lan00}, \cite{Lan04}).
Furthermore, Conjecture T respects birational equivalence and finite
morphisms.  In particular, Conjecture T holds for unirational surfaces
and Kummer surfaces.
Also, Conjecture T holds for any surface which lifts to a smooth
surface in characteristic zero with geometric genus $p_g = 0$.


\subsection*{Relationship with the Birch and Swinnerton-Dyer Conjecture}

The refined statement of the Birch and Swinnerton-Dyer (BSD)
conjecture over global fields of characteristic $p$ appears in
\cite[\S1]{Ta66}.

\begin{conj*}[BSD Conjecture]
Let $K$ be a global field of characteristic $p$, $A$ an abelian
$K$-variety, and $L^*(A,s)$ the $L$-function of $A$ with respect to a
suitably normalized Tamagawa measure.  Then $L^*(A,s)$ has an analytic
continuation to the complex plane and
$$
\lim_{s\to 1} \frac{L^*(A,s)}{(s-1)^{r}} = \frac{|\Sha(A)| \,
|\det(h_A)|}{|A(K)_{\mathrm{tors}}| \, |\hat{A}(K)_{\mathrm{tors}}|}
$$
where $h_A$ is the N\'eron--Tate height pairing on $A(K)$ modulo
torsion, $\hat{A}$ is the dual abelian variety, and $r = \rk{\Z}A(K)$
is the rank of the Mordell--Weil group.
\end{conj*}

Milne \cite{Mi68} proves the BSD conjecture for constant abelian
varieties.  Otherwise, there are a host of results on the BSD
conjecture for particular classes of abelian varieties over global
fields of characteristic $p$.  For a modern treatment of the BSD
conjecture for abelian varieties over number fields, see
Hindry--Silverman \cite[Appendix F.4]{hindry-silverman}.

Now, let $C$ be a geometrically connected, smooth, and projective
$k$-curve with function field $K=k(C)$ and suppose that there is a
proper flat morphism $f : X \to C$ with generic fiber $X_K$ a
geometrically connected, smooth, and projective $K$-curve.  Let $A$ be
the jacobian of $X_K$.

\begin{thm}[Artin--Tate Conjecture d] \label{thm:D}
Conjecture C for the $k$-surface $X$ is equivalent to
the BSD conjecture for the jacobian $A$ of the
$K$-curve $X_K$.
\end{thm}

Tate explains that ``this conjecture gets only a small letter d
as label, because it is of a much more elementary nature'' than
Conjecture C or the BSD conjecture, see \cite[\S4]{Ta66}.  The
Artin--Tate conjecture d is proved by Liu--Lorenzini--Raynaud
\cite{LLR05}, following a suggestion of Leslie Saper.
In fact, as shown in \cite[Proof of Theorem 1]{LLR05}, any
geometrically connected, smooth, and projective $k$-surface $X$ is
birational to another such surface $X'$ admitting a proper flat
morphism $f' : X' \to \mathbb{P}^1$ with geometrically connected,
smooth, and projective generic fiber.  In particular, up to birational
transformations, we may replace any surface by a surface fibered over
a curve to which Conjecture d applies.

\subsection*{Square order of the Brauer group}

If either group is assumed to be finite, then a precise relationship
between $|\!\Br(X)|$ and $|\Sha(A)|$ is given in
Liu--Lorenzini--Raynaud \cite[Th.~4.3]{LLR04}, \cite[Corollary
3]{LLR05} in terms of the local and global periods and indices of the
curve $X_K$.

\begin{thm}[Liu--Lorenzini--Raynaud {\cite[Cor.~3]{LLR05}}]
Let $f : X \to C$ be as above.  For each closed point $v$ of $C$,
denote by $K_v$ the completion of $K=k(C)$ at $v$, and by $\delta_v$
and $\delta_v'$ the index and period, respectively, of the $K_v$-curve
$X_{K_v}$.  Denote by $\delta$ the index of the $K$-curve $X_K$.
Assume that either $Br(X)$ or $\Sha(A)$ is finite, where $A$ is the
jacobian of $X_K$.  Then
$$
|\Sha(A)| \, \prod_{v \in C} \delta_{v} \delta_{v}' = |\!\Br(X)| \, \delta^2
$$
and $|\!\Br(X)|$ is a square.
\end{thm}

As explained in the introduction of \cite{LLR05}, the fact that the
Brauer group of $X$ has (conjecturally) square order is surprising
given the history of the subject.  Tate \cite{Ta66} (see also Milne
\cite[Th.~2.4]{Mi75}) exhibit a skew-symmetric bilinear pairing,
\begin{eqnarray*}
\Br(X) \times \Br(X) &\to& \Q/\Z \\
(x,y) &\mapsto& \mathrm{tr}(\tilde{x} \cup \beta\tilde{y})
\end{eqnarray*}
on the Brauer group of a surface over a finite field, defined by
taking lifts $\tilde{x}, \tilde{y} \in \H^2_{\text{\'et}}(X,\mu_n)$ of
$x, y \in \Brp[n](X)$, where $\beta : H_{\text{\'et}}^2(X,\mu_n) \to
H_{\text{\'et}}^3(X,\mu_n)$ is the coboundary map arising from the
exact sequence $1 \to \mu_n \to \mu_{n^2} \to \mu_n \to 1$, and
$\mathrm{tr} : H_{\text{\'et}}^5(X,\mu_n^{\ot 2}) \cong \Z/n\Z$ is the
arithmetic trace map induced from duality and the Hochschild--Serre
spectral sequence, as in \cite[Formula 5.4]{Ta66}.  This pairing is
non-degenerate on the prime-to-$p$ part of the Brauer group and has
kernel consisting of the maximal divisible subgroup of $\Br(X)$.  In
particular, if $\Br(X)$ is assumed to be finite, it must have order
either a square or twice a square.

In the late 1960s, Manin published examples of rational surfaces over
finite fields with Brauer group of order 2.  Finally, in 1996, Urabe
found a mistake in Manin's examples and proved that rational surfaces
always have Brauer groups of square order.  The proof that
$\Brp[2](X)$ is a square in \cite{LLR05} uses knowledge of
${}_2\Sha(A)$, via Theorem \ref{thm:D}.

The (conjectural) square order of $\Br(X)$ would immediately result
from the existence of a non-degenerate alternating pairing.  Of
course, this is only a nontrivial question on the 2-primary torsion
$\Br(X)(2)$ of the Brauer group.

\begin{prob}
Given a geometrically connected, smooth, projective surface $X$ over a
finite field $k$, is there a non-degenerate alternating pairing on
$\Br(X)(2)$ with values in $\Q/\Z$?
\end{prob}

This is addressed in Urabe \cite[Th.~0.3]{urabe:bilinear}, where
it is proved that the restriction of Tate's skew-symmetric form to
$\text{ker}(\Br(X) \to \Br(\ol{X}))$ is alternating.  As a corollary,
this settles the question whenever
$H_{\text{\'et}}^3(\ol{X},\Z_2(1))_{\text{tors}}$ is trivial, in
particular, this is the case for rational surfaces, ruled surfaces,
abelian surfaces, $K3$ surfaces, or complete intersections in
projective space.  The question is still open, for example, if $X$ is
an Enriques surface.  Also see Poonen--Stoll's \cite[\S11]{PoSt99}
possible strategy for finding a counterexample.

\section{The tensor product problem}\label{tensor.sec}

Both surveys \cite{Am:surv} and \cite{Salt:fd} address problems concerning division algebra
decomposability. 
We say an $F$-division algebra $D$ is \emph{decomposable} if there exist $F$-division
algebras $A$ and $B$ such that $D\simeq A\otimes_F B$.
All division algebras decompose into $p$-primary tensor factors, and 
all $p$-primary factors with equal period and index are indecomposable,
so the problem of when a division algebra decomposes
is only interesting for division algebras of unequal prime-power period and index.
Albert proved that any division algebra of period $2$ and index 4 is decomposable in \cite{Albert:deg4},
and for a long time all known
division algebras of unequal period and index were decomposable by construction.
However, in 1979 Saltman and Amitsur--Rowen--Tignol independently constructed 
nontrivially indecomposable division algebras, in \cite{Salt:indec} and \cite{ART}.
In 1991 Jacob produced indecomposable division algebras of prime period $p$ and index $p^n$ for any $n\ge 1$, 
except when $p=2$ and $n=2$ (\cite{Jacob:indec}).  
Thus by the time of Saltman's survey \cite{Salt:fd}, examples had been produced of every conceivable period-index
combination, though not in every characteristic.

One of the few remaining gaps in the list of \cite{Salt:fd}
was filled shortly afterwards by Karpenko,
who produced indecomposable division algebras (of a generic type) in any characteristic with 
any odd prime-power period $p^m$ and index $p^n$, for any $m\leq n$ (see also \cite{McK08}), 
and period $2$ and index $2^n$ for $n\geq 3$.
The odd degree examples used the index reduction formula \cite{SvdB} of Schofield and van den Bergh
and the theory of cycles on Severi--Brauer varieties (\cite{Karp:tor}),
and the $2$-power degree examples used his theorem \cite[Prop.~5.3]{Karp:cod} that
a division algebra $A$ of prime period is indecomposable
if the torsion subgroup of the Chow group of codimension-2 cycles 
on $A$'s Severi--Brauer variety is nontrivial.
The two problems on indecomposability asked in \cite{Salt:fd} were settled by
Brussel in \cite{Bru00} (Problem 6) and \cite{B:embed} (Problem 7), respectively.
However, Brussel's solution to Problem 6, which asked for an indecomposable division algebra of prime degree $p^n$
that becomes decomposable over a prime-to-$p$ extension, relied on the absence of a $p$-th
root of unity, and Problem 6 remains open for fields containing the $p$-th roots of unity.
Thus the search for a ``geometric'' obstruction to decomposability is still interesting.
For example, Karpenko's examples resulting from Karpenko's 
theorem are manifestly stable under prime-to-$p$ extension.

\begin{prob}\label{pr1} Is there a structural criterion for the tensor product
$A\otimes_F B$ of two $F$-division algebras $A$ and $B$ to have a zero divisor?
\end{prob}

For a fixed primitive $n$-th root of unity $\xi$, let
$(a,b)_n$ denote the symbol algebra of degree $n$ over $F$, generated by elements $u$ and $v$
under the relations $u^n=a$, $v^n=b$, $uv=\xi vu$.

\begin{prob}\label{pr2} Suppose $(a,b)_n$ and $(c,d)_n$ are symbol algebras
of degree $n$ over $F$, and $(a,b)_n \ot_F (c,d)_n$ has index $< n^2$.  
Are there necessarily symbol algebras $(a',b')_n$ and $(c',d')_n$ that share a common maximal 
subfield and such that $(a, b)_n \ot_F (c, d)_n \cong (a',b')_n \ot_F (c', d')_n$?
\end{prob}

One can also consider splitting fields instead of maximal
subfields:

\begin{prob}\label{pr3} If $\ind(A\ot_F B) < \ind (A) \ind(B)$,
what are the  common splitting fields for the $F$-division algebras $A$ and $B$?
\end{prob}

\Wik The first major result concerning the tensor product
question was Albert's theorem that the tensor product of central
quaternion algebras is a division algebra if and only if they do not have
a common maximal subfield (\cite[Th.~6]{Albert:deg4}). Using a similar proof,
Risman~\cite{Risman} showed that for any quaternion algebra $A$
and  any division algebra $B$, if $A \otimes B$ is not a
division algebra, then either $B$ contains a field $K$ such that
$A \otimes K$ is not a division algebra, or $A$ contains a
field $K$ such that $K \otimes B$ is not a division algebra.

In spite of Albert's theorem for quaternions,
Tignol and Wadsworth \cite[Prop.~5.1]{TW} produced for any odd integer $n$ an
example in which $A$ and $B$ are division algebras of degree $n$, 
$A \otimes_F B$ is not a division algebra, yet $A$ and $B$ have no common maximal subfields.
Saltman pointed out that for these examples, $A\otimes_F B^\text{\rm op}$ \emph{is}
a division algebra, and he noted that a better question to ask is therefore
whether one could have $A^{\otimes i} \ot B$ not a division algebra
for all $i$, with $A$ and $B$ still having no common subfields.
Mammone answered the question in \cite{Mamm},
producing examples $A\otimes_F B$ with $A$ of degree $n$ and $B$ of degree $n^2$ 
for any $n$.
Then Jacob and Wadsworth \cite[Th.~1]{JW93}  constructed examples of degree $p$ for
any odd prime $p$, using valuation-theoretic methods.
In \cite[Th.~II.1]{Karp:3}, Karpenko produced an essentially different family of examples,
also of equal odd prime degrees,
using the theory of the Chow group of Severi--Brauer varieties.

The examples of \cite{JW93} support an affirmative answer to Problem~\ref{pr2}. 
In fact, they involve symbol algebras 
$A = (a,\lambda)_p$ and $B = (c,d\lambda)_p$ of degree $n=p$
over a power series field $F := k((\lambda))$,
with $a,c,d\in k$, and 
$A^{\otimes i}  \otimes B \cong (a^ic,\lambda)_p \otimes (c,d)_p.$
This has index $p$ iff $k(\root p
\of {a^ic})$ splits $(c,d)_p$, which is the case if and only if
$(a^ic,\lambda)_p$ and $(c,d)_p$ have a common subfield.  

As for Problem~\ref{pr3}, it follows from Albert's and Risman's examples that if $A$ is a quaternion, 
$B$ has degree $2^n$, and $A\otimes B$
is not a division algebra, then $A$ and $B$ have a common splitting field of degree
less than or equal to $2^n$ over $F$.
The Mammone, Jacob--Wadsworth, and Karpenko examples
showed that the Albert and Risman results could not be generalized 
in the obvious way; in particular any common splitting field of the degree $p$
examples is divisible by $p^2$.
However, Karpenko proceeded to generalize the Albert--Risman results as follows.
Let $A$ be a central simple $F$-algebra of degree $p$, let $B_1,\dots,B_{p-1}$
be central simple $F$-algebras of degrees $p^{n_1},\dots,p^{n_{p-1}}$,
and let $n=n_1+\cdots+n_{p-1}$.
Suppose that for every $i=1,\dots,p-1$, $A\otimes_F B_i$ has zero divisors.
Then there exists a field extension $E/F$ of degree less than or equal to $p^n$
that splits all of the algebras $A,B_1,\dots,B_{p-1}$ (\cite[Th.~1.1]{Karp:3}).
In a similar vein, Krashen proved that if $A_1,\dots,A_{p^k}$ are central simple $F$-algebras
of degrees $p^n$ for a prime $p$,
and $A_1\otimes_F\cdots\otimes_F A_{p^k}$ has index $p^k$, then the $A_i$ are split by an \'etale
extension $E/F$ of degree $p^{n(p^k-1)}m$, where $m$ is prime-to-$p$ (\cite[Cor.~4.4]{Kra10}.  Problems such as these are connected with the notion of canonical dimension, see \cite{Karp:ICM} for a survey.

As noted above,
Karpenko's examples made use of the seminal index reduction formula due to Schofield and Van den Bergh from \cite{SvdB}.  Such formulas 
should play a role in helping to resolve Problem \ref{pr3}, which is framed quite generally.  For example,
there has been considerable recent interest in the special case:
\emph{If $A$ and $B$ have the same splitting fields \emph{(chosen from some class of extension fields)}, then are $A$ and
$B$ isomorphic?}  We address this in the next section.


\section{Amitsur Conjecture} \label{splitting}

As a central simple $F$-algebra $A$ is largely understood via
the study of its splitting fields (see \S\ref{background}),
it is natural to try to determine exactly how much information about $A$ is captured by
its splitting fields alone.


In \cite{Am:gen}, Amitsur constructed for any $F$-division algebra $A$ 
a {\it generic splitting field} $F(A)$, whose defining property is that
a field extension $L/F$ splits $A$ if and only if $F(A)$ has a place in $L$ (see \cite[VI.2]{Bou:commalg}).
The first field of this type was constructed by Witt in \cite{Witt34}, for quaternion algebras.
Amitsur proved that his $F(A)$ is the function field of $A$'s Severi-Brauer variety $SB(A)$,
first studied by Chatelet in \cite{Cha44}.
Nowadays, the generic splitting field of a central simple algebra is viewed as a special
instance of the generic splitting field of a (reductive) algebraic group, as described in \cite{KR}.
Amitsur proved:
\begin{thm} \label{Am.thm} {\textrm{\cite[\S9]{Am:gen}}}
Let $A$ and $B$ be central simple $F$-algebras.
\begin{enumerate}
\item $A$ and $B$ have the same splitting fields if and only if
$A$ is Brauer-equivalent to $B^{\ot i}$ for some integer $i$
relatively prime to the period of $A$.
\item If $F(A)$ is isomorphic to $F(B)$, then $A$ and $B$ generate the
same subgroup of the Brauer group.
\end{enumerate}
\end{thm}

He then posed the following converse of (2):

\begin{conj}[Amitsur's Conjecture]
If $A$ and $B$ have the same degree and generate the same subgroup of
the Brauer group, then $F(A)$ and $F(B)$ are isomorphic.
\end{conj}
The question was raised in \cite{Am:gen}, and appears as Question 6 in \cite{Am:surv}.

By definition, the generic splitting fields $F(A)$ and $F(B)$ are
isomorphic if and only if the Severi--Brauer varieties $SB(A)$ and
$SB(B)$ are $F$-birational.  It is not hard to show that if $A$ and $B$ are as in the conjecture, then $F(A)$ and $F(B)$ are stably isomorphic.

\Wik Amitsur was able to prove his conjecture in either the case where $A$ has a
maximal subfield that is cyclic Galois, or in the case where $A$ is
arbitrary and $B$ is $A^\op$ \cite{Am:gen}.  Roquette later extended these
results to the case where $A$ has a maximal subfield which is solvable
Galois \cite{Roq:iso}. Tregub, using very different techniques, proved
the result for
$A$ arbitrary of odd degree in the case where $B$ is equivalent to $A^{\ot 2}$ \cite{Tregub}. Krashen
generalized Roquette's results to the case where $A$ has a maximal
subfield whose Galois closure is dihedral (hence quadratic over the
maximal subfield), and also showed that the conjecture may be reduced
to considering the case where every nontrivial subfield of $A$ is a
maximal subfield \cite{Krashen:SBsemi}.

\smallskip

Theorem \ref{Am.thm}(1) says that the splitting fields almost determine the algebra.  But it may be too much to expect to know all of the splitting fields of a central simple algebra, and so we will pose a finer problem.  Note first that Theorem \ref{Am.thm}(2) says that if algebras $A$ and $B$ have degree $n$, then the fields $F(A)$ and $F(B)$ have transcendence degree $n - 1$.  It follows that \emph{if $A$ and $B$ have the same splitting fields of transcendence degree $< n$, then they generate the same subgroup of the Brauer group.}  
It is therefore natural
to ask about splitting fields of smaller transcendence degrees, or
even finite splitting fields.

\begin{prob} Determine whether two division algebras that have the same splitting fields in a given
class (say finite, or of transcendence degree at most one), necessarily generate the
same subgroup in the Brauer group.
\end{prob}

The smallest open case appears to be:

\begin{prob}
Determine whether two division algebras of degree $3$ that have the same splitting
fields of transcendence degree $\leq 1$ generate the same subgroup of the Brauer group.
\end{prob}

In the case of finite extensions, one has to be a bit more careful;
one may easily construct examples of central simple algebras $A$ and $B$ over
global fields which have all the same finite splitting fields, but
which do not generate the same cyclic subgroup of the Brauer group \cite[p.~149]{PrRap:weakly}. On
the other hand, it is easy to show that if $A$ and $B$ are \emph{quaternion}
algebras over a global field that have the same collection of
quadratic splitting fields, then $A$ and $B$ must
be isomorphic.  This motivates the following.

\begin{prob} Show that if $F$ is a field which is finitely generated over a
prime field or over an algebraically closed field, then any two
quaternion division algebras having the same finite splitting fields (or even
just the same maximal subfields) must be isomorphic.
\end{prob}

\Wik Garibaldi and Saltman \cite{GS:quats} showed that if $F$ has
trivial ``unramified Brauer group'' (in a sense including the
archimedean places), then any two quaternion division algebras over $F$ having
the same maximal subfields must be isomorphic. Rapinchuk and Rapinchuk proved similar results for period 2 division algebras in \cite{Raprap}.  Krashen and McKinnie
\cite{KraMcK} showed that if quaternion algebras over $F$ may be
distinguished by their finite splitting fields, then the same is true
over $F(x)$.

For algebras of higher degree, the situation for global fields
leads one to the following problem:

\begin{prob} Show that if $F$ is a field which is finitely generated over a
prime field or over an algebraically closed field, and $D$ is a
division algebra, then there are only finitely many other division
algebras which share the same collection of finite splitting fields.
\end{prob}

\Wik One may easily verify this for global fields. Krashen and
McKinnie \cite{KraMcK} showed that if this is true for a field $F$, it
remains true for the field $F(x)$.

\section{Other problems}

\subsection{Group admissibility and division algebras}

Let $F$ be a field and $G$ a finite group.
We say $G$ is \emph{$F$-admissible} if there exists a 
$G$-crossed product $F$-division algebra $A$.
\begin{prob} \label{sch.prob}
Determine which groups are $F$-admissible, for a field $F$.
\end{prob}
This problem is an extension of
the inverse Galois problem, and is especially interesting for $F=\Q$.
It appears as part of Problem 5 in \cite{Am:surv}.

The next problem, also known as the \emph{$\Q$-admissibility conjecture}, 
was first recognized in \cite{Schacher:th1}.
\begin{prob} \label{sch.conj}
Prove every Sylow-metacyclic group is $\Q$-admissible.
\end{prob}

\Wik
We discuss only the number field case of Problem~\ref{sch.prob}.
Schacher initiated the study of admissibility in his thesis, 
where he proved that for any finite group $G$ there exists a number field $F$
over which $G$ is $F$-admissible (\cite[Th.~9.1]{Schacher:th1}).
His key observation was that if $F$ is a global field, a group $G$ is $F$-admissible if and only if
(a) $G$ is the group for a Galois extension $L/F$, and (b) for every prime $p$
dividing $|G|$, there are two places $v_1$ and $v_2$ such that
$\Gal(L_{v_i}/F_{v_i})$ contains a $p$-Sylow subgroup of $G$ (\cite[Prop.~2.5]{Schacher:th1}).
It follows easily that a $\Q$-admissible group is \emph{Sylow-metacyclic}:
every Sylow subgroup $P$ of $G$ contains a normal subgroup $H$ such that
both $H$ and $P/H$ are cyclic (\cite[Th.~4.1]{Schacher:th1}).
This motivated Problem~\ref{sch.conj}.

Schacher proved all Sylow-metacyclic abelian groups are $\Q$-admissible
in \cite[Th.~6.1]{Schacher:th1},
and Gordon--Schacher proved that all metacyclic $p$-groups are $\Q$-admissible
in \cite{GS79}.
Liedahl extended this result to give necessary and sufficient conditions on a pair $(G,F)$ for
$G$ to be $F$-admissible, where $G$ is a metacyclic $p$-group and $F$ is a number field (\cite{Lie96}).
In 1983
Sonn settled Problem~\ref{sch.conj} for solvable groups, proving they are all $\Q$-admissible in \cite{Son83}.

To help organize the work for nonsolvable groups we make the following observations.
Schacher showed that $S_n$ is $\Q$-admissible if and only if $n\leq 5$ in
\cite[Th.~7.1]{Schacher:th1}.
He showed that the alternating group $A_n$ is not $\Q$-admissible
for $n\geq 8$ in \cite[Th.~7.4]{Schacher:th1}, but did not solve the problem
for $n=4,5,6$, and $7$, though
Gordon and Schacher proved $A_4=\PSL(2,3)$ is $\Q$-admissible in \cite{GS77}.
Schacher also noted that the groups $\PSL(2,p)$, which are simple for $p\geq 5$, 
are Sylow-metacyclic for $p$ a prime,
making them a natural target for attempts at Problem~\ref{sch.conj}.
Chillag and Sonn noted that
the central extensions $\SL(2,p^n)$ of $\PSL(2,p^n)$
are Sylow-metacyclic for $n\leq 2$ and $p\geq 5$,
and that a nonabelian finite simple Sylow-metacyclic group
belongs to the set $\{A_7, M_{11}, \text{$\PSL(2,p^n)$ for $n\leq 2$ and $p^n\neq 2,3$} \}$
(\cite[Lemma 1.5]{CS81}).
In 1993, Fein--Saltman--Schacher proved that any finite Sylow-metacyclic group that
appears as a Galois group over $\Q$ is $\Q(t)$-admissible (\cite{FSS93}).

Here is a list of known non-solvable
$\Q$-admissible groups and the papers in which the results appear:
$A_5=\PSL(2,4)=\PSL(2,5)$ (\cite{GS79});
$A_6=\PSL(2,9)$ and $A_7$, (\cite{FV87}, \cite{SS92});
$\SL(2,5)$, (\cite{FF90});
the double covers $\tilde A_6=\SL(2,9)$ and $\tilde A_7$, (\cite{Fei94});
$\PSL(2,7)$ and $\PSL(2,11)$, (\cite{AS01});
and $\SL(2,11)$, (\cite{Fei02}).
In some cases, these groups are known to be $F$-admissible in more general circumstances.
For example, 
the condition
``$2$ has at least two prime divisors in $F$ or $F$ does not contain $\sqrt{-1}$''
is necessary and sufficient for the $F$-admissibility 
of the groups $\SL(2,5)$, (\cite{FF90}); $A_6$ and $A_7$ (\cite{SS92}),
and $\PSL(2,7)$, (\cite{AS01}).
Feit proved $\PSL(2,11)$ is $F$-admissible over all number fields in \cite{Fei04}.
Problem~\ref{sch.conj} for $\PSL(2,13)$ remains open.

Interestingly, an analog of the $\Q$-admissibility conjecture was recently proved 
by Harbater, Hartmann, and Krashen for the function
field of a curve over a discretely valued field with algebraically closed residue field,
using patching machinery (\cite{HHK:subfields}).

\subsection{$SK_1$ of central simple algebras}

For a central simple $F$-algebra $A$, write $\SL_1(A)(F)$ for the elements of $A$ with reduced norm 1 and 
\[
\SK_1(A) := \SL_1(A)(F) / [A^\times, A^\times].
\]
This group depends only on the Brauer class of $A$ \cite[p.~160]{Draxl}.  If $A$ is not division, then $\SK_1(A)$ is the Whitehead group of the linear algebraic group $\SL_1(A)$ appearing in the Kneser--Tits Problem, see \cite{Ti:white} and \cite{Gille:KT}.

The basic question is: What is this group?  It is always abelian (obviously) and torsion with period dividing the index of $A$ \cite[p.~157, Lemma 2]{Draxl}.  It was for a long time an open question (called \emph{the Tannaka--Artin Problem}) whether it is always zero.   For example, Wang proved in \cite{Wang} that it is zero if the index of $A$ is square-free.  Then in \cite{Platonov:SK1}, Platonov exhibited a biquaternion algebra with nonzero $\SK_1$; re-worked versions of this example can be found also in \cite{KMRT} and \cite[\S24]{Draxl}.  Moreover, every countable abelian torsion group of finite exponent---e.g., a finite abelian group---is $\SK_1(A)$ for some algebra $A$ over some field.

One can also consider $\SK_1$ simultaneously over all extensions of $F$, which is connected with the structure of $\SL_r(A)$ as a variety:
\begin{thm}  For a central simple $F$-algebra $A$, $\SK_1(A \ot L) = 0$ for every extension $L/F$ if and only if $\SL_r(A)$ is retract rational for every (resp., some) $r \ge 2$.
\end{thm}
The ``if'' direction can be found in \cite[p.~186]{Vos}.  The full theorem amounts to combining the results from ibid.\ with \cite[Th.~5.9]{Gille:KT}.  Gille's theorem furthermore shows that the conditions are equivalent to the abstract group $\SL_r(A)(L)$ being ``projectively simple" for every extension $L/F$ and every $r \ge 2$.

The main problem from this point of view is to prove the following converse to Wang's result:

\begin{conj} [Suslin \protect{\cite[p.~75]{Suslin:SK1}}] If $\SK_1(A \ot L) = 0$ for every extension $L/F$, then the index of $A$ is square-free.
\end{conj}

The main known result on this conjecture is due to Merkurjev in \cite{Merk:genSK1} and \cite{M:SK12}: If 4 divides the index of $A$, then $\SK_1(A \ot L) \ne 0$ for some $L/F$.

\subsection{Relative Brauer groups}

Let $F$ be a  field with nontrivial Brauer group.  Which subgroups of the Brauer group of $F$ are \emph{algebraic} relative Brauer groups; i.e., of the form $\Br(K/F)$ for some \emph{algebraic} field extension $K$ of $F$?  Which subgroups of the Brauer group of $F$ are \emph{abelian} relative Brauer groups; i.e., of the form $\Br(K/F)$ for some \emph{abelian} field extension $K$ of $F$?

In particular, if $n>1$ when is the $n$-torsion subgroup $\Brp[n](F)$ of $\Br(F)$ an  algebraic (resp.\ abelian) relative Brauer group?

It is conjectured that if $F$ is finitely generated (and infinite), then $\Brp[n](F)$  is an abelian relative Brauer group for all $n$ if and only if $F$ is a global field.

\Wik  If $F$ is a global field, then $\Brp[n](F)$  is an abelian relative Brauer group for all $n$  (\cite{KSonn:ntor}; \cite{Popescu}; \cite{KSonn:ab}).  The conjecture above is that the converse holds.  So far the following is known \cite{PSW}: if $\ell$ is a prime different from the characteristic of $F$ and $F$ contains the $\ell^2$ roots of unity, then the conjecture holds for $n=\ell$, i.e., there is no abelian extension $K/F$ such that $\Brp[\ell](F)=\Br(K/F)$.

\subsection{The Brumer--Rosen Conjecture}

Amongst the first examples one sees, probably the fields $\R$ or iterated Laurent series $\R((x_1))((x_2))\cdots((x_n))$ are the only ones whose Brauer group is nonzero and finite.  This leads to the following naive question:
\begin{prob} \label{BR.prob}
For which natural numbers $n$ is there a field $F$ such that $|\!\Br(F)| = n$?
\end{prob}

Fein and Schacher \cite{FeinSch:BrQ} noted that every power of 2 can occur.  The following, finer conjecture would imply that these are the only possibilities.
Specifically,
for any field $F$, we can decompose $\Br F$ into a direct sum  of its $p$-primary components $(\Br F)_p$ and ask if the following holds:

\begin{conj}[Brumer--Rosen \protect{\cite{BrumerRosen}}] \label{BR.conj}
For a given field $F$ and each prime $p$, one of the following holds:
\begin{enumerate}
\renewcommand{\theenumi}{\alph{enumi}}
\item $(\Br F)_p = 0$.
\item $(\Br F)_p$ contains a nonzero divisible subgroup.
\item $p = 2$ and $(\Br F)_2$ is an elementary abelian $2$-group.
\end{enumerate}
\end{conj}

Some remarks:
\begin{enumerate}
\item In contrast with the case for fields covered by this conjecture, every finite abelian group can be obtained as the Brauer group of some \emph{commutative ring} \cite{Ford:finite}.
\item Every divisible torsion abelian group is the Brauer group of some field, see \cite{M:divisible}.  (Interestingly, this note does not appear in MathSciNet, is not listed in the journal's table of contents, and has apparently not been translated into English.)
\item $(\Br F)_2$ is an elementary abelian 2-group if and only if the
cup product map $? \cdot (-1) \!: \H^2(F, \mu_2) \ra \H^3(F,
\mu_2^{\tensor 2})$ is an injection \cite[Cor.~A4]{LLT}.
\end{enumerate}

``Many'' cases of the conjecture are known:
\begin{itemize}
\item \emph{if $p = \chr F$.}  In this case, $(\Br F)_p$ is $p$-divisible by Witt's Theorem \cite[p.~110]{Draxl}.
\item \emph{if $p = 2$ or $3$.} These are Theorem 3 and Corollary 2 in \cite{M:Brauer}.
\item \emph{if $p$ is odd and there is a division algebra of degree $p$ that is cyclic.} This is Theorem 2 in \cite{M:Brauer}.  This gives a connection between the Brumer--Rosen Conjecture and Problems \ref{cyclic} and \ref{Gen2}.
\item \emph{if $p = 5$ or $F$ contains a primitive $p$-th root of unity.} These follow from the previous bullet, because in these cases the $p$-torsion in $\Br F$ is generated by cyclic algebras of degree $p$.
\end{itemize}

Clearly, Merkurjev has largely settled the conjecture, in the same sense that the Merkurjev--Suslin Theorem settled a large portion of Problem \ref{Gen2} on generation by cyclic algebras.  The smallest open case is $p = 7$, and this case would be settled by a ``yes'' answer to Problem \ref{Gen7}.  Does there exist a field with Brauer group $\Z/7$?

We remark that Fein and Schacher's main result in \cite{FeinSch:BrQ} was that ``many'' of the possibilities predicted by Brumer--Rosen actually occur: every countable abelian torsion group of the form $D \oplus V$ for $D$ divisible and $V$ a vector space over $\F_2$ occurs as the Brauer group of some algebraic extension of $\Q$.  Efrat \cite{Efrat:finite} gives general criteria on the field $F$ for $(\Br F)_2$ to be finite and nonzero.  

\subsection{Invariants of central simple algebras}
Once one is interested in central simple $F$-algebras of degree $n$---equivalently, elements of $\H^1(F, \PGL_n)$---it is natural to want to know the invariants of $\PGL_n$ in the sense of \cite{GMS} with values in, say, mod-2 Galois cohomology $\H^i(-, \Z/2\Z)$ or the Witt ring $W(-)$.  (In this subsection, we assume for simplicity that the characteristic of $F$ does not divide $2n$.)  

Note that this problem is related to Problem \ref{center.prob}: If
$\PGL_n$ has a non-constant invariant that is unramified in the sense
of \cite[p.~87]{GMS}, then the center of generic matrices $Z(F, n)$
from \S\ref{center.sec} is not a rational extension of $F$, see ibid.,
p.~87.  There are no nonconstant unramified invariants with values in
$\H^i(-, \Gm)$ for $i \le 3$ by \cite{Sal85} (for $i = 2$) and
\cite{Salt:H3} (for $i = 3$).

This problem is also related to essential dimension as in
\S\ref{ed.sec}, because cohomological invariants provide lower bounds
on essential dimension.  Specifically: If $F$ is algebraically closed
and there is a nonzero invariant $\H^1(-, \PGL_n) \ra \H^i(-,C)$ for
some finite Galois module $C$, then $\ed(\PGL_n) \ge i$ by
\cite[p.~32]{GMS}.

\Wik The case $n = 2$ concerns quaternion algebras and the invariants
are determined in \cite[pp.~43, 66]{GMS}.  For $n$ an odd prime, there
are no nonconstant Witt invariants (for trivial reasons) and the
mod-$n$ cohomological invariants are linear combinations of constants
and the connecting homomorphism $\H^1(-, \PGL_n) \to \H^2(-, \mu_n)$,
see \cite[6.1]{G:lens}.

The case $n = 4$ is substantial, but solved.  Rost--Serre--Tignol
\cite{RST} construct invariants of algebras of degree 4, including one
that detects whether an algebra is cyclic.  (See \cite{Tig:PGL4} for
analogous results in characteristic 2.)  Rost proved that, roughly
speaking, the invariants from \cite{RST} generate all cohomological
invariants of algebras of degree 4---see \cite[\S3.4]{Baek:th} for details.

The next interesting case is $n = 8$.  Here, Baek--Merkurjev
\cite{BaekM} solve the subproblem of determining the invariants of
algebras of degree 8 and period 2, i.e., of $\GL_8/\mu_2$.  Beyond
this, not much is known.

\bibliographystyle{amsalpha}
\bibliography{ketura}
\end{document}